\documentclass[12pt,reqno,oneside]{amsart}
\usepackage{amsmath}
\usepackage{amssymb}
\usepackage{amsthm}
\usepackage{mathrsfs}
\usepackage{graphics}
\usepackage{eucal}
\usepackage{bm}
\usepackage{url}
\usepackage{hyperref}
\usepackage[capitalize,nameinlink,noabbrev]{cleveref}
\usepackage[normalem]{ulem}
\usepackage{color}
\usepackage[dvipsnames]{xcolor}

\newcommand{\hide}[1]{}
\usepackage{mathtools}
\usepackage{geometry}
\newtheorem{theorem}{Theorem}

\newtheorem{problem}{Problem}
\newtheorem{prop}{Proposition}[section]
\newtheorem{lemma}{Lemma}[section]
\newtheorem{corollary}{Corollary}

\newtheorem{remark}{Remark}[section]

\numberwithin{equation}{section}
\newtheorem*{thm}{Theorem}
\newcommand{\SL}{\operatorname{SL}}
\newcommand{\PGL}{\operatorname{PGL}}

\newcommand{\GL}{\operatorname{GL}}
\newcommand{\sgn}{\operatorname{sgn}}
\newcommand{\supp}{\operatorname{supp}}
\newcommand{\Eis}{\operatorname{Eis}}

\newcommand{\R}{\mathbb{R}}
\newcommand{\C}{\mathbb{C}}
\newcommand{\Q}{\mathbb{Q}}
\newcommand{\Z}{\mathbb{Z}}
\newcommand{\N}{\mathbb{N}}
\newcommand{\I}{\mathcal{I}}

\newcommand{\B}{\mathcal{B}}
\newcommand{\D}{\mathcal{D}}
\renewcommand{\d}{\,\mathrm{d}}
\renewcommand{\b}{\mathrm{b}}

\newcommand{\X}{\mathbb{X}}
\newcommand{\Dh}{\dim_{\mathcal{H}}}
\newcommand{\Dl}{\dim^{\star}_{\ell^1}}

\newcommand{\dl}{\dim_{\ell^1}}

\title{On Fourier asymptotics and effective equidistribution}

\author{Shreyasi Datta}
\address{Indian Statistical Institute, 8th Mile Mysore Road, RVCE Post, Bangalore 560059, India}
\email{shreyasi@isibang.ac.in}

\author{Subhajit Jana}
\address{Indian Statistical Institute, 8th Mile Mysore Road, RVCE Post, Bangalore 560059, India}
\email{s.jana@isibang.ac.in}

\thanks{The first author was in part supported by the Knut and Alice Wallenberg Foundation and also by the EPSRC grant EP/Y016769/1.}

\begin{document}

\begin{abstract}
We prove effective equidistribution of expanding horocycles in $\mathrm{SL}_2(\mathbb{Z})\backslash\mathrm{SL}_2(\mathbb{R})$ with respect to various classes of Borel probability measures on $\mathbb{R}$ having certain Fourier asymptotics. Our proof involves new techniques combining tools from automorphic forms and harmonic analysis.

In particular, for any Borel probability measure $\mu$, satisfying $\sum_{\mathbb{Z}\ni|m|\leq X}|\widehat{\mu}(m)| = O\left(X^{1/2-\theta}\right)$ with $\theta>7/64,$ our result holds. This class of measures contains convolutions of $s$-Ahlfors regular measures for $s>39/64$, and as well as, a sub-class of self-similar measures. Moreover, our result is sharp upon the Ramanujan--Petersson Conjecture (upon which the above $\theta$ can be chosen arbitrarily small): there are measures $\mu$ with $\widehat\mu(\xi) = O\left(|\xi|^{-1/2+\epsilon}\right)$ for which equidistribution fails.
\end{abstract}

\maketitle

\section{Introduction}

\subsection{Equidistribution of horocycle flow}

Dynamics on $\X:=\SL_2(\Z)\backslash\SL_2(\R)$, the unit cotangent bundle of the modular surface, is a central topic lying in the interface of homogeneous dynamics and representation theory and has numerous applications in various number theoretic problems. We consider the homogeneous space $\X$ equipped with the $\SL_2(\R)$-invariant probability measure $m_{\X}$. Let $g_y:=\left(\begin{smallmatrix}\sqrt{y}&\\&\sqrt{y^{-1}}\end{smallmatrix}\right)$ for $y>0$ be the geodesic flow and the corresponding expanding horocycle flow is given by $n(x):=\bigl(\begin{smallmatrix}1&x\\0&1\end{smallmatrix}\bigr)$ for $x\in\R$. Let $\mu$ denote a Borel probability measure on $\R$. Let $\mu_y:=\int\delta_{n(x)g_y}\d\mu(x)$ denote the probability measure on $\X$ supported on the horocycle of length of $y^{-1}$, that is, for $\phi\in C_c^\infty(\X)$ we write
\begin{equation}\label{eq:def-mu-y}
\mu_y(\phi):=\int\phi(n(x)g_y)\d\mu(x).
\end{equation}
We record the following (folklore) problem which is currently mostly wide open.

\begin{problem}\label{prob1}
    Given a Borel probability measure $\mu$ on $\R$, determine whether $\mu_y\xrightarrow{\text{weak-}^\ast}m_\X$ as $y\to 0$ with a polynomial order convergence rate, that is, if there exists an $\eta>0$ such that
    $$\mu_y(\phi)=m_\X(\phi)+O_{\mu,\phi}(y^\eta),\quad \phi\in C_c^\infty(\X),$$
    as $y\to 0$ where $m_\X(\phi):=\int\phi\d m_\X$.
\end{problem}

Let $\widehat\mu$ denote the Fourier transform of $\mu$; see\eqref{hatmu}. The purpose of this paper is to initiate a study of Problem \ref{prob1}, depending on the asymptotic behavior of $\widehat\mu$. Loosely speaking, we answer Problem \ref{prob1} affirmatively for $\mu$ with $\widehat\mu$ having
\begin{enumerate}
    \item\label{item_avergae} a certain polynomial decay \emph{on average}, not necessarily having a pointwise decay,
    \item\label{item_pointwise} arbitrarily slow polynomial decay, but having \emph{precise oscillatory asymptotics}.
\end{enumerate}
Interestingly, both of the above types contain many interesting classes of measures, as evidenced below.  

\subsection{Average Fourier decay}
Our first main result is an affirmative solution of Problem \ref{prob1} when $\mu$ satisfies property Item \ref{item_avergae}. We recall the definition of the \emph{Fourier $\ell^1$-dimension} of $\mu$ (also, see \eqref{def:lp-dim-mu}) as
$$\dim_{\ell^1}\mu:=1-\inf\left\lbrace s\ge 0\mid \lim_{X\to\infty}X^{-s}\sum_{|m|\le X}|\widehat\mu(m)|=0\right\rbrace.$$
We refer to see \cite[\S 1.4.1]{yu2021rational} and \cite[\S 2]{CVY24} for more details about this quantity.

\begin{theorem}\label{thm:fractal_withoutbase}
     Let $\mu$ be a Borel probability measure on $\R$ such that $$\dl\mu>\tfrac{39}{64}=0.609375.$$ 
     Then there exists an $\eta>0$ and $\ell\in\N$ so that for any $\phi\in C_\b^{2\ell}(\X)$ we have
     \begin{equation*}
     \mu_y(\phi)=m_\X(\phi)+ O_\mu\left(|y|^{\eta}S_{\infty,\ell}(\phi)\right),
     \end{equation*}
     as $y\to 0$.
\end{theorem}
We refer the reader to \S \ref{sobol} for the definitions of $C_\b^{2\ell}(\X)$ and $S_{\infty,\ell}(\cdot)$. The number $0.609375=\tfrac{39}{64}=\tfrac{1}{2}+\tfrac{7}{64}$ is related to the \emph{spectral gap} for $\X$; see \S\ref{sec:fourier-coefficient} for details. We point out that Theorem \ref{thm:fractal_withoutbase} is new in this setting even in its qualitative form, that is, without any rate of convergence.

\begin{remark}
    Note that the above theorem does not require any restriction of the geometric structure of the measure $\mu$, such as self-similarity. Moreover, we can choose any $0<\eta<\dl\mu-\tfrac{39}{64}$ in Theorem \ref{thm:fractal_withoutbase}.
\end{remark}

More generally, we have the following.

\begin{corollary}\label{cor:convol-infty}
    Let $\mu$ be a Borel probability measure on $\R$ with $\dl\mu>\tfrac{39}{64}$. Then there exist $\eta>0$ and $\ell>0$ such that for any $x_0\in \R,$ we have $$\int\phi(n(x_0+x)g_y)\d\mu(x)=\int_\X\phi\d m_\X + O_\mu(|y|^\eta S_{\infty,\ell}(\phi))$$
    uniformly in $x_0$. 
\end{corollary}

\subsection {Sharpness of Theorem \ref{thm:fractal_withoutbase}}\label{opti}

Quite interestingly, we note that the number $\tfrac{39}{64}$ in Theorem \ref{thm:fractal_withoutbase} is not an artifact of the method, but it is a reality check. Upon the assumption of the \emph{Ramanujan--Petersson Conjecture} (see the discussion after \eqref{eq:bound-hecke-cusp}) we can replace the number $\tfrac{39}{64}$ by $\tfrac{1}{2}$ in Theorem \ref{thm:fractal_withoutbase} -- this is \emph{sharp}. Indeed, for every $\epsilon>0$, there exists a measure $\mu_{\epsilon}$, constructed by Kaufman in \cite{Kauf81}, satisfying $\dl\mu_{\epsilon}\geq \frac{1}{2}-\epsilon$. On the other hand, $\mu_\epsilon$ is supported on $\mathcal{W}(\psi_{\epsilon+1})$, where $\psi_{\tau}(q):=q^{-\tau},\, q\in\N$ (see \S \ref{sec:Diop} for definitions). Thus, employing \cite[Theorem 9.1]{khalil2023random}, we see that $\mu_{\epsilon,y}$ \emph{fails} to equidistribute.

\vspace{8mm}

We now give the existence of a class of measures that satisfy the Fourier $\ell^1$-dimension condition in Theorem \ref{thm:fractal_withoutbase}. We refer to \S\ref{sec:measure-theory} for the undefined terminologies mentioned in the following. 

\begin{corollary}\label{cor:convol-2}
    Let $\mu_i$, for $i=1, 2$, be two Borel probability measures on $\R$ such that each $\mu_i$ is $s_i$-AD-regular (see \S\ref{sec:measure-theory} for definitions) for $s_i>0$. If $\frac{s_1+s_2}{2}>\tfrac{39}{64}$, then the convolution $\mu_1\ast\mu_2$ satisfies the conclusion of Theorem \ref{thm:fractal_withoutbase}.
\end{corollary}

Recently, Chow--Varj{\'u}--Yu in \cite{CVY24} find a robust way to compute Fourier $\ell^1$-dimensions for self-similar measures whose attractors are missing digit Cantor sets. In the following, $\Dh X$ denotes the Hausdorff dimension of a set $X$. Let $K_{b,D,x}$ denote a shifted missing digit Cantor set, as defined in \eqref{eq:shifted-IFS}. Let $\mu_{b,D,x}$ denote the $\frac{\log\#D}{\log b}$-dimensional Hausdorff measure on $K_{b,D,x}$; see \eqref{eq:haus-shifted-IFS}.

\begin{corollary}\label{intro_thm1}
For every $s>0.609375$, there exists an explicitly computable $b(s)\in\N$, as in \eqref{eqn:b(s)}, such that for any $b\ge b(s)$, any $x\in\R$, and any $D$ in arithmetic progression, with $\Dh K_{b,D}\geq s,$
the measure $\mu_{b,D, x}$ satisfies the conclusion of Theorem \ref{thm:fractal_withoutbase}.
\end{corollary}

In the case of $x\in\Q$ and $\Dh K_{b,D}>0.9992$, ($\Dh K_{b,D}>0.839$ for prime $b$) the conclusion of Corollary \ref{intro_thm1} was obtained in a breakthrough work \cite[Theorem C, Remark 1.4]{khalil2023random}.
However, Corollary \ref{intro_thm1} is the first result that affirmatively answers Problem \ref{prob1} for a class of \emph{irrational} self-similar measures on $\R$. As an easy example, let $b=450$ and $D=\{0,1,\cdots,446\}$, then $\dim_{\mathcal{H}}(K_{b,D})<0.9992$ and $\dl\mu_{b,D}>0.609375$ by \eqref{eqn:l1}.\footnote{This corollary was recently further improved by B{\'e}nard--He--Zhang \cite{benard2024khintchin} for any \textit{self-similar} measure on $\R$.}. Also, we note that $b(s)$ as given in \eqref{eqn:b(s)} is not optimal.

Our method of proof is completely different than that of \cite{khalil2023random}; see \S \ref{sec:idea} for details and comparison with other works.

\subsection{Ideas and comparison with other work}\label{sec:idea}

Let $\phi\in C_c^\infty(\X)$ with $m_\X(\phi)=0$ and $\mu$ be a Borel probability measure on $\R$. To prove Theorem \ref{thm:fractal_withoutbase} for certain Borel probability measure $\mu$ we need to show that $\mu_y(\phi)\ll |y|^\eta$ for some $\eta>0$. Note that $\phi(n(\cdot)g)$ is $\Z$-invariant. After Plancherel we write
\begin{equation*}
    \mu_y(\phi)=\sum_{m\in\Z}\widehat\phi_y(m){\widehat\mu(m)},\quad\widehat\phi_y(m):=\int_{\R/\Z}\phi(n(x)a(y))e(-mx)\d x.
\end{equation*}
Integrating by parts in the above integral with $\left(\begin{smallmatrix}0&1\\0&0\end{smallmatrix}\right)$ in the universal enveloping algebra of $\mathfrak{sl}_2(\R)$ one checks that $\widehat\phi_y(m)$ is supported essentially on $|m|\le |y|^{-1}$. Thus we essentially need to show that
\begin{equation}\label{eq:truncation-baby}
    \sum_{|m|\le |y|^{-1}}\widehat\phi_y(m){\widehat\mu(m)} \ll |y|^\eta,\quad\eta>0.
\end{equation}
At this point, we use \emph{spectral decomposition} of $L^2(\SL_2(\Z)\backslash\SL_2(\R))$. Thus we need to show \eqref{eq:truncation-baby} for $\phi$ varying over a suitable orthonormal basis in every \emph{standard} automorphic representation for $\SL_2(\Z)$, with a certain Sobolev-type uniformity as the representations vary. This is a new input that has not been explored previously in case of singular measures and is \emph{crucial} for our result. Using the local representation theory of $\SL_2(\R)$ and \emph{spectral gap} of $\X$ we prove
\begin{equation}\label{eq:spectral-gap}
    \sup_m|\widehat\phi_y(m)| \ll |y|^{1/2-\theta},\quad \theta>\frac{7}{64}.
\end{equation}
Consequently, Theorem \ref{thm:fractal_withoutbase} follows from the assumption on $\dim_{\ell^1}\mu$.

Now we briefly mention that our proof has a certain similar flavour as the same of \cite[Theorem 3.2]{Ven-sparse} and \cite[Proposition 4.2]{LM}. More precisely, Venkatesh in \cite[Theorem 3.2]{Ven-sparse} obtained a general bound of  $\widehat{\phi}_y(m)$  for \textit{any test function} $\phi$ using the effective decay of matrix coefficients. However, his method only yields a bound of the form
$$\sup_m|\widehat\phi_y(m)| \ll |y|^{\delta}$$
for some unspecified $\delta>0$; see \cite[Remark 3.1]{Ven-sparse}. This bound is much weaker than \eqref{eq:spectral-gap}, which is required for Theorem \ref{thm:fractal_withoutbase}. In particular, the sharpness of Theorem \ref{thm:fractal_withoutbase}, as described in \S\ref{opti}, seems to be not possible with the approach of \cite{Ven-sparse} -- with the approach of \cite{Ven-sparse} one will need $\dim_{\ell^1}\mu$ to be sufficiently close to $1$, as opposed to $1/2$. Our novelty in this aspect is to apply spectral decomposition and using finer information about the Fourier coefficient of automorphic forms.

On the other hand, it seems to us that the \emph{sufficient large dimensionality} condition of the measure $\mu$, as required in \cite[Proposition 4.2]{LM} (also, \cite{khalil2023random} and other works proving Ratner-type effective equidistribution in the first steps of these proofs \footnote{Although these works need large dimensionality condition at the beginning, their bootstrap process eventually leads to effective equidistribution without any restriction on dimension.}), is of a different nature than our condition on $\dl\mu$. We think that upon \emph{extra} geometric or regularity condition on $\mu$ one may expect to connect the notions of the average growth of $\widehat\mu$ (phase space) with the dimension of $\mu$ (physical counterpart). 

Finally, beyond Theorem \ref{thm:fractal_withoutbase}, there could be measures for which $\dim_{\ell^1}\mu$ is very small. In this case, if we have more information about $\widehat\mu$ then we may still prove polynomially effective equidistribution. For instance, if $\widehat\mu(m)$ asymptotically behaves as $e(m\alpha)m^{-\delta}$ for any $\delta>0$, we can exploit a \emph{cancellation of the additive twisted Fourier coefficients}, e.g. by Jutila \cite{jutila1987tifr}, to win. This is the main content of Theorem \ref{thm:smooth} in the next subsection.

\subsection{Pointwise Fourier decay}
Our next main theorem answers Problem \ref{prob1} for measures that satisfy Item \ref{item_pointwise}. 

\begin{theorem}\label{thm:smooth}
Let $\mu$ be a Borel probability measure on $\R$ such that there exist a $K\in\N$, sequences $\{\delta_j\}_{j=1}^K\subset(0,\tfrac{1}{2}]$ and $\delta_0>\frac{1}{2}$, and complex numbers $\{\alpha_j\}_{j=1}^K\subset\R$ and $\{\beta_j\}_{j=1}^K\subset\C$ such that
$$\widehat\mu(\xi)= \sum_{j=1}^K|\xi|^{-\delta_j}\beta_j{e(\xi\alpha_j)}+O_\mu\left(|\xi|^{-\delta_0}\right), \quad e(z):=\exp(2\pi i z),$$
as $|\xi|\to\infty$. Then there exists an $\ell\in\N$ such that for any $0<\eta<\min\{\tfrac{1}{2},\delta_0,\tfrac{\delta_1}{2}\dots,\tfrac{\delta_K}{2}\}$ and $\phi\in C_\b^{2\ell}(\X)$ we have
\begin{equation*}
    \mu_y(\phi)=m_{\X}(\phi)+ O_{\eta,\mu}\left(|y|^{\eta}S_{\infty,\ell}(\phi)\right),
\end{equation*}
as $y\to 0$.
\end{theorem}

As an immediate corollary (which can be deduced more directly using Ramanujan--Petersson Conjecture on average, (\emph{cf}.\ Lemma \ref{lem:error-term-additive-twist})) we record the following.

\begin{corollary}
    Let $\mu$ be a Borel probability measure on $\R$ such that there exists $\delta>0$ such that 
    $$\vert\widehat{\mu}(\xi)\vert \ll \vert \xi\vert^{-(\frac{1}{2}+\delta)}.$$ Then there exists an $l\in \N$ such that for any $0<\eta< \frac{1}{2}$ and $\phi\in C_\b^{2\ell}(\X)$ we have
\begin{equation*}
    \mu_y(\phi)=m_{\X}(\phi)+ O_{\eta,\mu}\left(|y|^{\eta}S_{\infty,\ell}(\phi)\right),
\end{equation*}
as $y\to 0$.
\end{corollary}

As a \emph{proof of concept} for Theorem \ref{thm:smooth}, using it we answer Problem \ref{prob1} for the push-forward of the Lebesgue measure by a non-constant analytic map, namely Corollary \ref{thm: function equi}. However, we are not claiming any originality here -- Corollary \ref{thm: function equi} can be deduced more directly (that is, not going via Theorem \ref{thm:smooth}).

Let $f:\R\to\R$ be a non-constant real analytic function and $w$ be a compactly supported non-negative $L^1$-normalized smooth function on $\R$. Let $\mu^{w,f}$ be the Borel probability measure defined by $f_\star(w\circ\mathrm{Leb})$ where $\mathrm{Leb}$ denotes the Lebesgue measure, that is
$$\mu^{w,f}(h):=\int_\R h(f(x))w(x)\d x,\quad h\in C(\R).$$
One can check that $\mu^{w,f}$ is an absolutely continuous measure with possibly \emph{non-continuous} density.
One may analyse the asymptotics of the Fourier transform of $\mu^{w,f}$ via the \emph{method of stationary phase}; see Proposition \ref{prop: Stationary phase}, which also shows that Theorem \ref{thm:smooth} is not vacuous.

Let $Z(f)$ denote the zero set of $f$. Note that as $f$ is analytic and $w$ is compactly supported the set $Z(f')\cap\supp(w)$ is finite (counted with multiplicity). We define
$$k_{f,w}:=\max\{\text{order of vanishing of }f'\text{ at any }z\in Z(f')\cap\supp(w)\}.$$

\begin{corollary}\label{thm: function equi}
There exists $\ell\in\N$ such that for any $0<\eta<\tfrac{1}{2(k_{f,w}+1)}$ and $\phi\in C_\b^{2\ell}(\X)$ we have 
\begin{equation*}
    \mu^{w,f}_y(\phi)=m_\X(\phi)+ O_{\eta}\left(|y|^{\eta}S_{\infty,\ell}(\phi)\right),
\end{equation*}
as $y\to 0$.
\end{corollary}

One may check that the rate of equidistribution in Corollary \ref{thm: function equi} can not be improved without assuming the \emph{Riemann Hypothesis} (\emph{cf}.\ the proof of \cite{strombergsson2004uniform} for $f(x)=x$). On higher rank groups the related problem is significantly difficult and, in general, open; see \cite{Shah09,YangP20,Yang23,CY24} for recent major works.

\subsection{Background} Problem \ref{prob1} has a long history when $\mu$ is the normalized Lebesgue measure supported on $[0,1]$. In this case, Problem \ref{prob1} has been investigated by numerous people over the last decades, starting with Selberg and Zagier; see \cite{sarnak1981asymptotic} and \cite{strobergsson2013deviation,strombergsson2004uniform,flaminio2003invariant, Mar_thesis,Bur90, SarU15} that make the equidistribution result by Dani--Smillie \cite{DaniSmillie} (and also \cite{Fus73}) effective.

For absolutely continuous measures, it is known that equidistribution holds as in Problem \ref{prob1}, but nothing can be said, a priori, about the rate of convergence. In fact, the convergence could happen arbitrary slowly if the density function (coming from Radon--Nikodym theorem) does not have enough regularity. Kleinbock--Margulis, in \cite[Theorem 1.1]{KM12}, answered Problem \ref{prob1} affirmatively for absolutely continuous measures with smooth compactly supported densities in the space of matrices. Recently, Bj{\"o}rklund--Gorodnik in \cite{BG23} extended \cite{KM12} to continuous densities, among other things.

Problem \ref{prob1} is significantly difficult when the measure is not absolutely continuous. Such problem emerges from deep works by Kleinbock--Margulis \cite{KM98} and  Kleinbock--Lindenstrauss--Weiss \cite{KLW}. In \cite{KLW}, the authors show that if $\mu$ is a \textit{friendly} measures (see \cite[\S 2]{KLW} for definition)  then any weak-{$^\ast$} limit of $\frac{-1}{\log\varepsilon}\int_{\varepsilon}^1 \mu_y \d^{\times} y$ as $\varepsilon\to 0$ is a probability measure. The class of friendly measures contain the Hausdorff measures on the missing digit Cantor sets (described below).
The celebrated work Simmons--Weiss \cite{SW19} implies that for certain \emph{self-similar} measures $\mu$, one has
\begin{equation}\label{eq:average-prob1}
    \lim_{\varepsilon\to 0}\,\frac{-1}{\log \varepsilon}\int_{\varepsilon}^1\mu_y\d^\times y=m_\X,
\end{equation}
(in weak-{$^\ast$} sense) improving \cite{KLW} in this case. Note that \eqref{eq:average-prob1} is an average version of Problem \ref{prob1} in its qualitative form. In a recent breakthrough, Khalil--L{\"u}thi \cite{khalil2023random} answered Problem \ref{prob1} for a class of measures that are not necessarily absolutely continuous. In particular, \cite{khalil2023random} considered general rational \emph{iterated function system} (IFS) in higher dimensions with restrictions coming from contraction ratios and probability vectors associated to the self-similar measure.

\section{Application in Diophantine Approximation}\label{sec:Diop} 

In this section, we focus on an application of Theorem \ref{thm:fractal_withoutbase} in Diophantine approximation. Given a non-increasing monotonic positive function $\psi:\N\to\R_{+}$, we denote 
\begin{equation*}
    \mathcal{W}(\psi):=\{x\in[0,1]\mid\vert qx-p\vert<\psi(q) \text{ for infinitely many }q\in\N\text{ and }p\in\Z\}.
\end{equation*}
 Khintchine in 1926, \cite{Khint26} shows that 
\begin{equation}\label{Khincthine}
    \text{Leb}(\mathcal{W}(\psi))=\begin{cases}
0 & \text{ if }  \sum \psi(q)<\infty,\\
1 & \text{ if }  \sum \psi(q)=\infty.
    \end{cases}
\end{equation}
The following question is vaguely posed in \cite{KLW}.
\begin{problem}\label{Prob2}
    For which measures $\mu$, does the analogue of \eqref{Khincthine} hold?
\end{problem}
In the last two decades, there has been a significant development in this question for $\mu$ being a singular measure; see \cite{KM98, BKM,BBKM,B12,huang,BY23, KLW, W01, einsiedler2011diophantine, SW19,  PV, LSV07, ACY23, PVZZ22}.

A complete answer to Problem \ref{Prob2} is given in \cite{khalil2023random}, for self-similar measures $\mu$ associated to rational irreducible IFS satisfying open set condition under a certain extra condition. This extra condition in case of the natural missing digit measures (see the end of \S\ref{sec:measure-theory} for definition) amounts to restrict the Hausdorff dimension of the missing digit Cantor set to be very close to $1$. 
 
With a different approach, in a groundbreaking work \cite{yu2021rational}, Yu answers the convergence part of \eqref{Khincthine} for a class of measures $\mu$ with $\dim_{\ell^1}\mu>\frac{1}{2}$, namely:
 
 \begin{thm}\cite[Theorem 1.5]{yu2021rational}
     Let $\mu$ be a Borel probability measure with $\dl\mu>\frac{1}{2}$. Also, let $\psi$ be a non-increasing monotonic function with $\sum\psi(q)<\infty$. Then $\mu(\mathcal{W}(\psi))=0$.
 \end{thm}
 
 For the special function $\psi(q)=\frac{1}{q\log\log q}$, Yu in \cite[Theorem 1.5]{yu2021rational} also proves divergence for self-similar measure with attractor being a missing digit Cantor set. However, his methods do not yield the divergence part in general. In this paper, we address the divergence part, extending the result of \cite{yu2021rational}. As a consequence, we extend the result of Khalil--L{\"u}thi \cite[Theorem A]{khalil2023random}. 

\begin{theorem}\label{intro_thm3}
   Let $\mu$ be a self-similar measure whose underlying attractor is a shifted missing digit Cantor set as described in \S\ref{sec:measure-theory}. If $$\dl\mu>\frac{39}{64},$$ then for any $\psi$ non-increasing monotonic function, we have 
    \begin{equation*}
    \mu(\mathcal{W}(\psi))=
1 \quad \text{ if } \quad  \sum \psi(q)=\infty. 
\end{equation*}
In particular, measures $\mu_{b,D,x}$ considered in Corollary \ref{intro_thm1} satisfy the above conclusion.
\end{theorem}

We remark that assuming the Ramanujan--Petersson Conjecture we can improve Theorem \ref{intro_thm3} by weakening the hypothesis to $\dl\mu>\tfrac{1}{2}$. We also note that Theorem \ref{thm:fractal_withoutbase} is not enough to imply Theorem \ref{intro_thm3}. One needs a stronger effective equidistribution result Proposition \ref{prop:fractal} with any starting point, not only the identity.

\section{Fourier Theoretic Preliminaries}\label{sec:measure-theory}

Let $\mu$ be a Borel probability measure on $\R$. Given two Borel probability measures $\mu_1$ and $\mu_2$ we denote their \emph{convolution} by
$$\mu_1\ast\mu_2(f) = \int f(x+y)\d\mu_1(x)\d\mu_2(y),\quad f\in C(\R).$$
We define the Fourier transform of $\mu$ by
\begin{equation}\label{hatmu}\widehat{\mu}(\xi):=\int e(\xi x)\d\mu(x);\quad e(z):=\exp(2\pi i z), z\in\C.\end{equation}
If $\mu$ is compactly supported then $\widehat{\mu}$ is a bounded Lipschitz continuous function. Note that $\widehat{\mu_1\ast\mu_2}=\widehat{\mu}_1\widehat{\mu}_2$.
It follows from the classical Riemann--Lebesgue lemma that any absolute continuous measure (w.r.t the Lebesgue measure) $\widehat\mu(\xi)\to 0$ as $\vert \xi\vert\to\infty$. However, it is often extremely difficult to quantify the decay rate of $\widehat\mu$ in general.

On the other hand, when $\mu$ is not absolutely continuous,  $\widehat\mu$ may not have any decay.
For instance, if $\mu$ is the Hausdorff measure supported on the middle-third Cantor set then $\widehat\mu(\xi)\nrightarrow 0$, i.e., $\widehat\mu$ has no decay. We refer readers to \cite{Varju_ICM, Tuo23, Pablo_Survey19, strichartz1990selfsimilar} for surveys on behaviour of Fourier transform in case of various singular measures.

In this paper, we mainly consider the Fourier decay in an $\ell^p$-average sense. Following \cite[Definition 1.2]{yu2021rational}, we define, 
\begin{equation}\label{def:lp-dim-mu}
    \dim_{\ell^p}\mu:=1-\inf\left\{\ell \mid\lim_{X\to\infty}X^{-\ell}\sum_{|m|\le X}|\widehat{\mu}(m)|^p = 0\right\}.
\end{equation}
The set on the right-hand side above is non-empty as trivially any $\ell>1$ belongs to this set. Moreover, this definition is equivalent to those in \cite[Definition 1.2]{yu2021rational}. Indeed, this follows from the fact for any sequence $\{A(X)\}_{X\in\N}$ of positive reals
$$A(X) = o_\epsilon(X^\epsilon)\, \forall\epsilon>0\iff A(X) = O_\epsilon(X^\epsilon)\, \forall\epsilon>0.$$
Using Cauchy--Schwarz inequality and definitions, as noted in \cite[Lemma 1.4]{yu2021rational}, we have
\begin{equation*}
    \tfrac{1}{2}\dim_{l^2}\mu\leq \dl\mu\le \dim_{l^2}\mu.
\end{equation*}
For $s>0$ we call a Borel measure $\mu$ to be \emph{$s$-AD regular} if there exists a $C>1$ such that for any $x\in\supp(\mu)$ and sufficiently small $r>0$ one has
$$C^{-1}r^s\le\mu((x-r,x+r))\le Cr^s.$$
By \cite[\S 3.8]{mattila2015fourier}, if $\mu$ is $s$-AD regular then, 
\begin{equation}\label{eqn:l2=s}
\dim_{\ell^2}\mu=s.
\end{equation}
We define a variant of the above Fourier $\ell^p$-dimension, following \cite[Definition 2.2]{yu2023missing}. For $1\le p <\infty$ we define
\begin{equation}\label{def:lp-dim-mu-star}
    \dim^{\star}_{\ell^p}\mu:=1-\inf\left\{\ell \mid\lim_{X\to\infty}X^{-\ell}\sup_{0\le \theta\le 1}\sum_{|m|\le X}|\widehat{\mu}(m+\theta)|^p = 0\right\}.
\end{equation}
It is immediate to check that 
\begin{equation}\label{eqn:relation_bet_dimen}
    \dim^{\star}_{\ell^p}\mu\leq \dim_{\ell^p}\mu.
\end{equation}
We remark that, for $p=1$ a self-similar measure whose attractor is a missing digit Cantor set (see below), the above two dimensions will coincide; see Lemma \ref{lemma:two dim same}.

\vspace{5mm}

We refer the readers to  \cite{mattila2015fourier} for definition of \textit{self-similar} measure, \emph{iterated function system} (IFS) and \textit{open set condition (OSC)} of an IFS. A self-similar measure comes with an IFS and a probability vector $\lambda.$ If an IFS $\mathcal{F}$ has OSC, then for $\lambda_j=\rho_j^s,$ where $\rho_j$'s are contraction ratios, the corresponding self-similar measure $\mu$ is the \emph{$s$-dimensional Hausdorff measure} restricted to the \emph{attractor} $K$ of $\mathcal{F}$ and $\mu$ is $s$-AD regular; see \cite[Theorem 4.14]{mattila1995geometry}.

Let $b\geq 3$ and $\varnothing\neq D\subset\{0,\cdots, b-1\}$ with $\# D=l\ge 2$. Let $\mathcal{F}_{b,D}$ denote the IFS $\{f_i(x):=\frac{x+i}{b} \mid i\in D\}$. The IFS $\mathcal{F}_{b,D}$ satisfies OSC and the attractor of $\mathcal{F}_{b,D}$ is $K_{b,D}$, the \emph{missing digit Cantor set} with base $b$ and (appearing) digits from $D$. Given any probability vector, there exists an unique self-similar measure whose attractor is $K_{b,D}$. We refer to them as self-similar measures for which the attractor is a missing digit Cantor set. In particular, when the probability vector is uniform probability vector, we call the self-similar measure to be natural missing digit measure $\mu_{b,D}.$ Note $\mu_{b,D}$ is also the $s=\frac{\log l}{\log b}$-dimensional Hausdorff measure restricted on $K_{b,D}$ (which we call by natural missing digit measure).

For $x_0\in\R$ it can be easily checked that $K_{b,D,x_0}:=K_{b,D}+x_0$, called as \emph{shifted missing digit Cantor set}, is the attractor of the IFS
\begin{equation}\label{eq:shifted-IFS}
    \mathcal{F}_{b,D,x_0}:=\left\{f_i(x):=\frac{x+i}{b}+x_0(1-\tfrac{1}{b})\mid i\in D\right\}
\end{equation}
and the convolution
\begin{equation}\label{eq:haus-shifted-IFS}
    \mu_{b,D,x_0}:=\mu_{d,D}\ast\delta_{x_0}
\end{equation}
is the self-similar measure on $K_{b,D,x_0}$ associated with uniform probability vector. It is easy to check that $\mathcal{F}_{b,D,x_0}$ satisfies OSC.

\section{Automorphic Preliminaries}

\subsection{Convention}

The letter $\epsilon$ will denote placeholders for a small positive number. In particular, the exact value of $\epsilon$ may change from line to line. We will make the same convention with the \emph{order of the Sobolev norm} (denoted by $d$ or $\ell$) which is thought as an implicit positive integer. Moreover, we will write $\ll x^{O_P(1)}$ to denote $\ll x^d$ for some unspecified positive $d$ depending on $P$ as $x\to\infty$.

\subsection{Groups and measures}

Let $G:=\PGL_2(\R)$, $\Gamma:=\PGL_2(\Z)$ and $[G]:=\Gamma\backslash G$. Throughout the paper we will identify $\X$ with $[G]$ by means of, e.g., \cite[Lemma 3.1]{khalil2023random}. We denote 
$$n(x):=\begin{pmatrix}1&x\\&1\end{pmatrix},\quad a(y):=\begin{pmatrix} y&\\&1\end{pmatrix}.$$
Via the above identification we rewrite $\mu_y$, as in \eqref{eq:def-mu-y}, for a Borel probability measure $\mu$ as
$$\mu_y(\phi)=\int\phi(n(x)a(y))\d\mu(x),\quad\phi\in C(\X).$$
Let $K:=\mathrm{SO}_2(\R)$, which is a compact subgroup of $G$. We will abbreviate $\exp(2\pi i z)$ by $e(z)$ for any $z\in\C$.

We fix measures $\d x$ on $N:=\{n(x)\mid x\in\R\}$ and $\d^\times y:=\frac{\d y}{|y|}$ on $A:=\{a(y)\mid y\in\R^\times\}$. Correspondingly, we can write a Haar measure on $G$ in the Iwasawa coordinates $g=n(x)a(y)k$ as $\d g=\d x\frac{\d^\times y}{|y|}\d k$ where $\d k$ is the probability Haar measure on $K$. We use $m_{[G]}$ (which is the same as $m_\X$) to denote the invariant probability measures on on $[G]$.

\subsection{Sobolev norm}\label{sobol}

We refer to \cite[Chapter 2]{bump1997automorphic} for details. Let $\{X,Y,Z\}$ be an orthonormal basis of $\mathrm{Lie}(G)$ where the orthonormality is with respect to the Killing form on $\mathrm{Lie}(G)$ and $Z$ is a non-zero element in $\mathrm{Lie}(K)$. We define a Laplacian of $G$ by the formula
\begin{equation}\label{def:diff-op}    
    \D:=\mathrm{Id}-X^2-Y^2-Z^2=\mathrm{Id}-C_G+2C_K
\end{equation}
where $C_G:=X^2+Y^2-Z^2$ is the Casimir operator on $G$ and $C_K:=-Z^2$ is the Casimir operator on $K$

Let $\pi$ be a unitary representation of $G$. It is known that (see, e.g., \cite[proof of Theorem 3]{nelsona1959analytic}) $\D\vert_{\pi}$ is a densely-defined, self-adjoint, invertible, positive definite, second-order elliptic differential operator. For $d\in\Z_{\ge 0}$ we define the $d$'th Sobolev norm on $\pi$ by
$$S_{2,d}(v) := \|\D^d v\|_\pi,\quad \text{smooth }v\in\pi $$
and extend the definition to $\pi$ by density of the smooth vectors in $\pi$.

The Casimir operator $C_G$ lies in the center of the universal enveloping algebra of $G$. Thus if $\pi$ is also irreducible then $C_G$ acts on $\pi$ by a scalar $\nu_\pi$. On the other hand, if $v\in\pi$ is a $K$-type vector (i.e.\ $K$ acts on $v$ by a character) of weight $k\in\Z$ then $C_K$ also acts on $v$ by a scalar $\nu_k$. Consequently, it follows from \eqref{def:diff-op} that $\D$ acts on a weight-$k$ vector $v\in\pi$ by a scalar. Thus a $\D$-eigenbasis of $\pi$ is the same as a $K$-isotypic basis of $\pi$.

\begin{lemma}\label{lem:trace-class-pi}
    For every $\ell>0$ there exists an $\ell'>0$ such that we have $$\mathrm{trace}\left(\D^{-\ell'}\vert_\pi\right)\ll (1+|\nu_\pi|)^{-\ell}$$
    for any unitary irreducible $\pi$.
\end{lemma}

\begin{proof}
    Let $\{v_k:k\in\Z\}$ be an orthonormal $K$-type basis of $\pi$ where $v_k$ is a weight $k$-vector. Thus, from the above discussion it follows that
    $$\mathrm{trace}\left(\D^{-\ell'}\vert_\pi\right) = \sum_{k\in\Z}(1-\nu_\pi+2\nu_k)^{-\ell'}.$$
    We readily check that $\nu_k$ is a positive constant multiple of $k^2$. On the other hand, from \cite[Proposition 2.5.4]{bump1997automorphic}, it follows that if $\pi$ is a principal series with parameter $\mu$ (so that $|\Re(\mu)|\le \tfrac{1}{2}$) or a discrete series of weight $q$ (so that $0<q\le |k|$) then $\nu_\pi$ is a positive constant multiple of $\mu^2$ or $q(2-q)$, respectively. Thus in either case the above sum is bounded by
    $$\ll \sum_{k\in\Z}(1+|\nu_\pi|)^{-\frac{\ell'}{2}}(1+|k|)^{-\ell'}.$$
    Making $\ell'$ sufficiently large we conclude.
\end{proof}

Finally, for $0\le \ell\le \infty$, we define
$$C_\b^{2\ell}([G]):=\{\phi\in C^{2\ell}([G])\mid S_{\infty,\ell}(\phi):=\|\D^\ell\phi\|_{L^\infty([G])}<\infty\}.$$
Note that for every $\ell>0$ and for every $\phi\in C_\b^{2\ell}([G])$ we have
$$S_{\infty,\ell}(\phi)\ll\mathcal{S}_{\infty,\ell}(\phi)$$
where $\mathcal{S}_{\infty,\ell}$ is the Sobolev norm defined in \cite[eq.(3.6)]{khalil2023random}.

\subsection{Automorphic forms and representations}

We refer to \cite{cogdell1990arithmetic} for details. In this section, $\pi$ will denote a standard non-trivial unitary automorphic representation for $\Gamma$, that is, $\pi$ is infinite dimensional and appears in the spectral decomposition of $L^2([G])$; see Lemma \ref{lem:spectral-decomposition}. Then $\pi$ is either cuspidal or a unitary Eisenstein series. By $\varphi\in\pi$ (cuspidal or an Eisenstein series) we will denote an automorphic form on $[G]$.

\subsubsection{Unitary structure}
If $\pi$ is cuspidal then we fix a $G$-invariant inner product on $\pi$ by 
$$\|\varphi\|^2_\pi:=\|\varphi\|^2_{[G]}=\int_{[G]}|\varphi(g)|^2\d m_{[G]}.$$

If $\pi$ is unitary Eisenstein then any $\varphi\in\pi$ is of the form
$$\varphi=\Eis(f),\quad f\in \I(s):=\mathrm{Ind}_{NA}^G |\cdot|^s,\quad s\in i\R;$$
where for $\Re(s)>\tfrac{1}{2}$ we define 
$$\Eis(f):=\sum_{\gamma\in N\cap\Gamma\backslash\Gamma}f(\gamma\cdot).$$
The sum converges absolutely for $\Re(s)>\tfrac{1}{2}$ and can be meromorphically continued to all $s\in\C$.
Note that $f$ satisfies
\begin{equation}
    f(n(x)a(y)g) = |y|^{\frac{1}{2}+s} f(g),\quad x\in \R, y\in\R^\times,g\in G.
\end{equation}
Thus any such $f$ is determined by $f\vert_K$ by Iwasawa decomposition. If $s\in i\R$ then $\I(s)$ is unitary. In this case, we fix a $G$-invariant inner product on $\I(s)$ by
$$\|f\|^2_{\I(s)}:=\int_K|f(k)|^2\d k$$
and a $G$-invariant inner product on $\pi$ by
$$\|\Eis(f)\|_\pi = \|f\|_{\I(s)}.$$
Through out the paper we only work with $f\in\I(s)$ such that $f\vert_K$ is independent of $s$. In this paper, if not mentioned otherwise, we assume that $f\vert_K$ is $s$-independent and a $K$-type vector. In this case, $\Eis(f)$ is holomorphic in the region $0\le\Re(s)<\tfrac{1}{2}$.

\subsubsection{Constant term}
We define the constant term of $\varphi$ by 
$$\varphi_0(g):=\int_{\R/\Z}\varphi(n(x)g)\d x.$$
If $\varphi$ is cuspidal then $\varphi_0=0$. If $\varphi=\Eis(f)$ then (see \cite[eq.(6)]{jana2021joint})
\begin{equation}\label{eq:constant-term}
    \varphi_0(g)=f(g)+\frac{\zeta(2s)}{\zeta(1+2s)}M(s)f(g),
\end{equation}
where $\zeta$ is the Riemann zeta function and $M(s)$ is the standard intertwiner mapping $\I(s)\to\I(-s)$, defined by
$$M(s)f(g):=\int_{\R}f(wn(x)g)\d x,\quad w:=\begin{pmatrix}&1\\1&\end{pmatrix}.$$
The above integral converges absolutely for $\Re(s)>0$ and has meromorphic continuation to all of $\C$.

\begin{lemma}\label{lem:constant-term-expression}
    Let $\varphi=\Eis(f)$ for $f\in\I(s)$. Then for all $x\in \R$, $y\in \R^\times$, and $k\in K$ we have
    $$\varphi_0(n(x)a(y)k)=\varphi_0(a(y)k)  =|y|^{\frac{1}{2}}\left(|y|^sf(k)+\frac{\zeta(2s)}{\zeta(1+2s)}|y|^{-s}M(s)f(k)\right).$$
    In particular, there exists a $d>0$ so that for $\Re(s)=0$
    $$\varphi_0(n(x)a(y)k)\ll |y|^{\frac{1}{2}}S_{2,d}(\varphi)$$
    uniformly in $x$ and $k$.
\end{lemma}

\begin{proof}
    The formula of $\varphi_0$ and left $N$-invariance of it follow immediately from \eqref{eq:constant-term}, and the facts that $f\in\I(s)$ and consequently, $M(s)f\in\I(-s)$.
    To see the next estimate, we first use the functional equation of the Riemann zeta function:
    $$\frac{\zeta(2s)}{\zeta(1+2s)}=\frac{\Gamma_{\R}(1-2s)}{\Gamma_\R(2s)}\frac{\zeta(1-2s)}{\zeta(1+2s)},\quad \Gamma_\R(s):=\pi^{-s/2}\Gamma(s/2).$$
    Note that
    $$\left|\frac{\zeta(1-2s)}{\zeta(1+2s)}\right|=1,\quad \Re(s)=0.$$
    Then the claim follows from the fact that
    $$f(k),\quad \frac{\Gamma_\R(1-2s)}{\Gamma_\R(2s)}M(s)f(k)\ll S_{2,d}(\varphi);$$
    see \cite[eq.(7),\,\S2.2]{jana2021joint} and the discussion around it.
\end{proof}

\subsubsection{Fourier coefficients}\label{sec:fourier-coefficient}
For any cuspidal or unitary Eisenstein series $\pi\ni\varphi$, one has the Fourier expansion
\begin{equation}\label{eq:fourier-expansion}
    \varphi(g) - \varphi_0(g) = \sum_{m\neq 0}\frac{\lambda_\pi(|m|)}{\sqrt{|m|}}W_\varphi(a(m)g).
\end{equation}

Let $\pi$ be cuspidal and $\varphi\in\pi$. In this paper, we always use the normalized $\varphi$ with $\lambda_\pi(1)=1$.
In this case, $\lambda_\pi$ are the Hecke eigenvalues attached to $\pi$. The \emph{Whittaker function} $W_\varphi$ in this case is given by
$$W_\varphi(g):=\int_{\R/\Z}\varphi(n(x)g)e(-x)\d x.$$
The expression \eqref{eq:fourier-expansion} follows from Shalika's multiplicity one theorem; see \cite[\S4.1 and \S6.2]{cogdell1990arithmetic}.
We have the following bound
\begin{equation}\label{eq:bound-hecke-cusp}
\lambda_\pi(|m|)\ll |m|^{\vartheta},\quad \vartheta >\frac{7}{64},
\end{equation}
for all $m\in\Z_{\neq 0}$ and uniformly in $\pi$; see \cite[Proposition 2]{kim-sarnak}.
The Ramanujan--Petersson Conjecture predicts that the above estimate of $\lambda_\pi$ can be improved to $|m|^\epsilon$ for any $\epsilon>0$. However, this is only known when $\pi$ corresponds to a modular form (that is, a discrete series representation).

If $\pi$ is an Eisenstein series with parameter $s\in\C$ and $\varphi=\Eis(f)$ then the Whittaker function $W_\varphi$ is given by
$$W_\varphi(g) := W_f(g):=\int_\R f(wn(x)g)e(-x)\d x$$
which converges absolutely for $\Re(s)>0$ and can be analytically continued to all of $\C$. In particular, when $f$ varies over a family such that $f\vert_K$ is $s$-independent (so called \emph{flat family)} then $W_f$ is entire in $s$. In this case, we have
\begin{equation}\label{eq:formula-hecke-eis}
    \lambda_\pi(|m|)= \frac{|m|^{-s}\tau_{2s}(|m|)}{\zeta(1+2s)},\quad \tau_z(m):=\sum_{d\mid m}d^z,
\end{equation}
for all $m\in\Z_{\neq 0}$; see \cite[\S2.1]{jana2021joint} for details of the proof.
Consequently, for $\Re(s)=0$ it follows from the lower bound of $\zeta$ on $\Re(s)=1$ that (see \cite[pp.50-51]{Tit86} and \cite[Proposition 2.5.4]{bump1997automorphic})
\begin{equation}\label{eq:bound-hecke-eis}
\lambda_\pi(|m|)\ll_\epsilon \left(|m|(1+|s|)\right)^\epsilon \asymp \left(|m|(1+|\nu_\pi|)\right)^\epsilon.
\end{equation}
Finally, the Ramanujan--Petersson Conjecture is known \emph{on average} for any cuspidal or unitary Eisenstein series. That is,
\begin{equation}\label{eq:average-GRC}
    \sum_{m\le X}|\lambda_\pi(m)|^2 \ll_\epsilon X^{1+\epsilon}(1+|\nu_\pi|)^\epsilon;
\end{equation}
see \cite[Proposition 19.6]{DFI02}.

\subsubsection{Functional equation of the Eisenstein series}
We record the functional equation of the Eisenstein series. It follows from (see \cite[eq.(3.10)]{blomer2024local}, cf. \cite[Proposition 4.5.9]{bump1997automorphic}) that
\begin{equation}\label{eq:whittaker-intertwiner-invariance}
W_{f} = \frac{\Gamma_\R(1-2s)}{\Gamma_\R(2s)}W_{M(s)f} =W_{M^\ast(s) f},\quad M^\ast(s):=\frac{\Gamma_\R(1-2s)}{\Gamma_\R(2s)}\dot M(s).
\end{equation}
Moreover, by Schur's lemma and the above it follows that
$$M^\ast(-s)\circ M^\ast(s)=\mathrm{Id}.$$
Thus from the Fourier expansion of Eisenstein series, namely, \eqref{eq:fourier-expansion}, \eqref{eq:constant-term}, and \eqref{eq:formula-hecke-eis}, and the functional equation of the Riemann zeta function it follows that
\begin{multline*}
    \zeta(1+2s)f(g) + \zeta(2s)M(s)f(g) +\sum_{m\neq 0}\frac{|m|^{-s}\tau_{2s}(|m|)}{\sqrt{|m|}}W_{f}(a(m)g)\\
    =\zeta(1-2s)M^\ast(s) f(g) + \zeta(-2s) M(-s)\circ M^\ast(s)f(g) + \sum_{m\neq 0}\frac{|m|^{s}\tau_{-2s}(|m|)}{\sqrt{|m|}}W_{M^\ast(s)f}(a(m)g)
\end{multline*}
Hence, we obtain
\begin{equation}\label{eq:functional-eqn-eis}
    \zeta(1+2s)\Eis(f) = \zeta(1-2s)\Eis(M^\ast(s) f) = \zeta(2s)\Eis(M(s)f)
\end{equation}
for any $s\in\C$.

\subsection{Whittaker functions}\label{sec:whittaker-function}

In this section, $\pi$ will denote an abstract unitary irreducible representation of $G$. We call $\pi$ to be \emph{generic} if $\pi$ has a unique $G$-invariant embedding into $$\mathrm{Ind}_N^G\,e(\cdot):=\{W:G\to\C\text{ smooth }: W\text{ satisfies }\eqref{eq:unipotent-equaivariance}\};$$
\begin{equation}\label{eq:unipotent-equaivariance}
W(n(x)g) = e(x)W(g),\quad \forall x\in\R, g\in G.
\end{equation}
The functions $W$ are known as Whittaker functions and the image of $\pi$ inside $\mathrm{Ind}_N^G\,e(\cdot)$ under the above embedding is known as the Whittaker model of $\pi$. In this paper, we always identify a generic representation and its Whittaker model.

\subsubsection{Mellin Theory}
From the theory of local $\GL(2)\times\GL(1)$ Hecke zeta integral (see, e.g., \cite[Chapter 8]{cogdell1990arithmetic}) we know that for $\varepsilon\in\{0,1\}$ and $s\in\C$ the zeta integral
\begin{equation}\label{def:zeta-integral}
    Z(s,\varepsilon,W):=\int_{\R^\times}W(a(y))|y|^s\sgn(y)^{\varepsilon}\d^\times y
\end{equation}
converges absolutely for $\Re(s)>-\frac{1}{2}$ and hence defines a holomorphic function in this region. The same also follows from the growth of $W(a(y))$, namely, for all $j\in\Z_{\ge 0}$ we have
\begin{equation}\label{eq:rapid-decay}
    (y\partial_y)^jW(a(y)k) \ll_{j,A} |y|^{-A}S_{2,d}(W),\quad y\in \R^\times, k\in K;
\end{equation}
where $A$ can take any value in $(-\frac{1}{2},\infty)$. Here $d$ only depends on $j$ and $A$. The above follows from \cite[Proposition 3.2.3]{MV} temperedness of $\pi$ (at the archimedean place) due to Selberg. Finally, we record the functional equation of the local $\GL(2)\times\GL(1)$ zeta integral (see \cite[eq.(8.1), Proposition 8.1]{cogdell1990arithmetic})
\begin{equation}\label{eq:local-functional-equation}
    Z(s,\varepsilon,W)=Z(-s,\varepsilon,W(\cdot w))\gamma(\tfrac{1}{2}-s,\pi)
\end{equation}
where $\gamma$ is a certain complex meromorphic function (implicitly depending on $\varepsilon$) satisfying
\begin{equation}\label{eq:bound-gamma-factor}
    \gamma(\tfrac{1}{2}-s,\pi)\ll_{\Re(s)} (1+|s|)^{2\Re(s)}\nu_\pi^{O(1)}
\end{equation}
as long as $s$ is a fixed distance away from any pole of the $\gamma$-function; see \cite[Proposition 9.5--9.6]{cogdell1990arithmetic}.

\subsubsection{Unitary structure}
We fix a $G$-invariant unitary structure on the Whittaker model of $\pi$ by 
$$\|W\|^2_\pi:=\int_{\R^\times}|W(a(y)|^2\d^\times y.$$
It is known that cuspidal representations and unitary Eisenstein series $\pi$ of $G$ are generic. For any such $\varphi\in\pi$ with Whittaker function $W_\varphi$ it follows from Schur's lemma that (see \cite[\S5.2-5.3]{cogdell1990arithmetic})
$$\|\varphi\|^2_\pi=C_\pi\|W_\varphi\|^2_\pi.$$

If $\pi$ is cuspidal then from \eqref{eq:average-GRC} and \cite[Theorem 0.2]{HoffLock94} it follows that
$$C_\pi^{-1}\ll_\epsilon (1+|\nu_\pi|)^\epsilon$$
for any $\epsilon>0$.

If $\pi$ is a unitary Eisenstein series with parameter $s\in i\R$ then $C_\pi=1$ which follows from the fact that
$$\|W_f\|_{\I(s)} = \|f\|_{\I(s)};$$
see, e.g., \cite[eq.(3.11)]{blomer2024local}. The above combining with \eqref{eq:whittaker-intertwiner-invariance} implies that
\begin{equation}\label{eq:intertwiner-isometry}
\|f\|_{\I(s)}=\|M^\ast(s)f\|_{\I(-s)}.
\end{equation}

Finally, if $\varphi$, cuspidal or unitary Eisenstein, is a $\D$-eigenfunction if and only if $W_\varphi$ is the same with the same eigenvalue. Thus the above estimates of $C_\pi$ imply that $$S_{2,d}(W_\varphi)\ll_\epsilon (1+|\nu_\pi|)^\epsilon S_{2,d}(\varphi)\ll S_{2,d}(\varphi)$$
for any $\D$-eigenfunction $\varphi$. From now on, if not mentioned otherwise, we assume that $\varphi$ is a $\D$-eigenfunction and use the above without anymore explanation.

\begin{lemma}[Truncation]\label{lem:truncation}
    Fix $\sigma>1$. For every $N>0$ sufficiently large there exists a $d>0$ so that for every automorphic form $\varphi$ we have
    $$(\varphi-\varphi_0)(n(x)a(y)) = \sum_{0\neq |m|\le |y|^{-\sigma}}\frac{\lambda_\pi(|m|)}{\sqrt{|m|}}W_\varphi(a(my))e(mx)+O\left(|y|^NS_{2,d}(\varphi)\right)$$
    as $y\to 0$.
\end{lemma}

\begin{proof}
    We start with the Fourier--Whittaker expansion of $\varphi$ is in \eqref{eq:fourier-expansion}. Using unipotent equivariance of the Whittaker function as in \eqref{eq:unipotent-equaivariance} we obtain
    $$(\varphi-\varphi_0)(n(x)a(y)) = \sum_{m\neq 0}\frac{\lambda_\pi(|m|)}{\sqrt{|m|}}W_\varphi(a(my))e(mx).$$
    Thus it is enough to show that the contribution of the terms for $|m|>|y|^{-\sigma}$ in the above sum is $O\left(|y|^NS_{2,d}(\varphi)\right)$ for all large $N$.
    
    We recall from \eqref{eq:rapid-decay} the bound of the Whittaker function
    $$W_\varphi(a(my))\ll_A |my|^{-A}S_{2,d}(\varphi)$$
    where $A$ will be chose at the end of the proof. From \eqref{eq:bound-hecke-cusp} and \eqref{eq:bound-hecke-eis} we obtain the bound of Hecke eigenvalues
    $$\frac{\lambda_\pi(|m|)}{\sqrt{|m|}}\ll (1+|\nu_\pi|)^\epsilon,\quad m\neq 0.$$
    Thus using triangle inequality and partial summation we obtain
    $$\sum_{|m|\ge |y|^{-\sigma}}\frac{\lambda_\pi(|m|)}{\sqrt{|m|}}W_\varphi(a(my))e(mx)\ll S_{2,d}(\varphi)\sum_{|m|\ge |y|^{-\sigma}}|my|^{-A}\ll |y|^{A(\sigma-1)-\sigma}S_{2,d}(\varphi).$$
    Making $A$ sufficiently large we conclude.
\end{proof}

\begin{lemma}\label{lem:l1-norm-bound}
    There exists a $d>0$ such that for any unitary automorphic form $\varphi$ we have
    $$\|\varphi\|_1 \ll S_{2,d}(\varphi)$$
    and
    $$\varphi(n(x)a(y)k)\ll_\epsilon \max\left(\sqrt{|y|}\,,\,|y|^{-1-\epsilon}\right)S_{2,d}(\varphi)$$
    for all $y\neq 0$, $k\in K$, and uniformly in $x\in\R$.
\end{lemma}

The second assertion is essentially the Sobolev embedding as stated in \cite[S2a]{MV}. However, we give a proof for completion.

\begin{proof}
    Applying the Fourier--Whittaker expansion \eqref{eq:fourier-expansion} and unipotent equivariance of Whittaker function \eqref{eq:unipotent-equaivariance} for any $g = n(x)a(y)k$ we write
    $$(\varphi-\varphi_0)(g) = \sum_{m\neq 0}\frac{\lambda_\pi(|m|)}{\sqrt{|m|}}W_\varphi(a(my)k)e(mx).$$
    Using the bounds \eqref{eq:bound-hecke-cusp}, \eqref{eq:bound-hecke-eis}, \eqref{eq:rapid-decay} and  working as in the proof of Lemma \ref{lem:truncation} with $A=1+\epsilon$ we bound the above by
    $$S_{2,d}(\varphi)\sum_{m\neq 0}|my|^{-1-\epsilon}\ll |y|^{-1-\epsilon}S_{2,d}(\varphi)$$
    for some fixed $d>0$. Finally, applying Lemma \ref{lem:constant-term-expression} we obtain
    $$\varphi(g)\ll \max\left(\sqrt{|y|}\,,\,|y|^{-1-\epsilon}\right)S_{2,d}(\varphi),\quad g=n(x)a(y)k.$$
    This proves the second assertion.

    To prove the first assertion we fix the Siegel domain
    $$S:=\{g=n(x)a(y)k\mid |x|\le\tfrac{1}{2},\, y\ge\tfrac{\sqrt{3}}{2},\, k\in K\}\supseteq [G].$$
    Then we obtain
    $$\|\varphi\|_1\le \int_S |\varphi|\ll S_{2,d}(\varphi)\int_{|x|\le \frac{1}{2}}\d x\int_{|y|\gg 1}\sqrt{|y|}\frac{\d^\times y}{|y|}\int_K \d k \ll S_{2,d}(\varphi)$$
    as claimed.
\end{proof}

\begin{remark}\label{rem:non-unitary}
    For $0\le\Re(s)<\tfrac{1}{2}$ and $f\in\I(s)$ it follows from the proof of Lemma \ref{lem:l1-norm-bound} and \eqref{eq:constant-term} that $\Eis(f)$ is $L^1$-integrable on $[G]$.
\end{remark}

\subsection{Spectral decomposition}

For a unitary representation $\pi$ (similarly, $\I(s)$) of $G$ by $\B(\pi)$ we denote an orthogonal basis of $\pi$ consisting of $K$-type vectors so that if $\pi$ is cuspidal then we normalize $\lambda_\pi(1)=1$ and if $\pi$ is Eisenstein then we normalize $\lambda_\pi$ as in \eqref{eq:formula-hecke-eis}. For the rest of the paper, by $\B(\pi)$ we mean a basis of $\pi$ with the properties as above, unless otherwise mentioned.

\begin{lemma}\label{lem:spectral-decomposition}
    There exist a sufficiently large $\ell>0$ such that for any $\phi\in C_\b^{2\ell}([G])$ and any $g\in G$ we have
\begin{equation*}
    \phi(g)=\int_{[G]}\phi\d m_{[G]}+\sum_{\pi\text{ cuspidal}}\sum_{\varphi\in\B(\pi)}\frac{\langle \phi,\varphi\rangle_{[G]}}{\|\varphi\|^2_\pi}\varphi(g)\\
    +\int_{\Re(s)=0}\sum_{f\in\B(\I(s))}\frac{\langle \phi,\Eis(f)\rangle_{[G]}}{\|f\|^2_{\I(s)}}\Eis(f)(g)\frac{\d s}{4\pi i}.
\end{equation*}
The right-hand side converges absolutely and uniformly on compacta. Moreover, in each summand the sum over $\B$ is invariant under a choice of orthogonal basis $\B$.
\end{lemma}

Lemma \ref{lem:spectral-decomposition} is essentially given in \cite[Lemma 2]{bruggeman205spectral}. However, in \cite{bruggeman205spectral} the test function $\phi$ is chosen to be inside $C^\infty([G])\cap L^2([G])$. Here we give a sketch to show that one may choose the test function from $C_\b^{2\ell}([G])$. The ideas of the proof are from \cite[\S2.6.5]{MV}.

\begin{proof}
    We start with the $L^2$-version of the equality given in the statement of the lemma, which can be found in \cite[Theorem, Chapter 1]{cogdell1990arithmetic}.

    Let $\pi$ be cuspidal and $\varphi\in\B(\pi)$. Thus $\varphi$ is also an eigenvector of $\D$ with, say, eigenvalue $\lambda_\varphi$. Integrating by parts with respect to $\D$ for $l$ times we obtain
    $$\langle\phi,\varphi\rangle_{[G]}=\lambda_\varphi^{-l}\langle\D^l\phi,\varphi\rangle_{[G]}\ll \lambda_\varphi^{-l}\|\D^l\phi\|_\infty\|\varphi\|_1.$$
    Applying Lemma \ref{lem:l1-norm-bound} we obtain
    $$\frac{\langle\phi,\varphi\rangle_{[G]}}{\|\varphi\|_\pi}\ll\lambda_\varphi^{-l+d}\|\D^l\phi\|_\infty.$$
    Working similarly, for $f\in\I(s)$ we obtain
    $$\frac{\langle \phi,\Eis(f)\rangle_{[G]}}{\|f\|_{\I(s)}}\ll\lambda_f^{-l+d}\|\D^l\phi\|_\infty$$
    where $\lambda_f$ is the $\D$-eigenvalue of $f$ (note that $\Eis(f)$ is a $\D$-eigenfunction if and only if $f$ is).
    Once again applying Lemma \ref{lem:l1-norm-bound} we have
    $$\varphi(g)\ll_g S_{2,d}(\varphi),\quad g\in [G].$$
    Thus for a fixed compact set $\Omega\subset [G]$ and $g\in\Omega$ the right-hand side of the claimed estimate in the lemma is bounded by
    $$\ll_\Omega \|\phi\|_1+S_{\infty,l}(\phi)\left(\sum_{\pi\text{ cuspidal}}\sum_{\varphi\in\B(\pi)}\lambda_\varphi^{-l+2d}+\int_{\R}\sum_{f\in\B(\I(it))}\lambda_f^{-l+2d}\d t\right).$$
    The above $\varphi$-sum (resp.\ $f$-sum) can be realized as $\mathrm{trace}\left(\D^{-l+2d}\vert_\pi\right)$ (resp.\ $\mathrm{trace}\left(\D^{-l+2d}\vert_{\I(it)}\right)$) as $\D$ is positive definite. We make $l$ sufficiently large. Applying Lemma \ref{lem:trace-class-pi} for cuspidal $\pi$ and Weyl's law (see, e.g., \cite[\S2.6.5 Corollary]{MV}) that
    $$\#\{\pi\text{ cuspidal }\mid |\nu_\pi|\le X\}\ll X^{O(1)},$$
    we conclude that the above $(\pi,\varphi)$-sum converges absolutely. Similarly, applying Lemma \ref{lem:trace-class-pi} for $\I(s)$ and noting that $\int_\R(1+|t|)^{-2}\d t<\infty$ we conclude that the above $t$-integral and $f$-sum converge absolutely.
    
    Thus the right-hand side of the claimed equation converges absolutely, hence defines a continuous function of $g$. Thus the pointwise decomposition follows.
\end{proof}

\section{Equidistribution with Average Fourier Decay}

\begin{lemma}\label{lem:fourier-transform-basis}
    Let $\pi\ni\varphi$ be cuspidal or a unitary Eisenstein series, $x_0\in \R/\Z$, and $q\in\N$. Let $\mu$ be a Borel probability measure on $\R$. Then for any $\sigma>1$ there exists a $d>0$ such that we have $$\int\varphi(n(x_0+\tfrac{x}{q})a(y))\d\mu(x)=\sum_{0\neq|m|\le |y|^{-\sigma}}\frac{\lambda_\pi(|m|)}{\sqrt{|m|}}W_\varphi(a(my))e(mx_0)\widehat{\mu}(m/q) +O\left(|y|^{\frac{1}{2}}S_{2,d}(\varphi)\right)$$
    as $y\to 0$.
\end{lemma}

\begin{proof}
    We use the equation in Lemma \ref{lem:truncation} and integrate both sides in $x$ with respect to $\mu$. As $\mu$ is a probability measure and $\varphi_0$ is left $N$-invariant (as in Lemma \ref{lem:constant-term-expression}) we obtain that
    \begin{multline*}
        \int\varphi(n(x_0+\tfrac{x}{q})a(y))\d\mu(x)=\varphi_0(a(y)) \\
        + \sum_{0\neq|m|\le |y|^{-\sigma}}\frac{\lambda_\pi(|m|)}{\sqrt{|m|}}W_\varphi(a(my))e(mx_0)\widehat{\mu}(m/q) +O\left(|y|^{\frac{1}{2}}S_{2,d}(\varphi)\right).
    \end{multline*}
    We conclude the proof after an application of the estimate in Lemma \ref{lem:constant-term-expression}.
\end{proof}

Recall the definition of $\Dl\mu$ from \eqref{def:lp-dim-mu-star}.

\begin{lemma}\label{lem:averaging-mu-hat-q}
    Let $\mu$ be a Borel probability measure and $\ell<\Dl\mu$. Then for $q\in\N$ and $X\ge 1$ we have
    $$\sum_{|m|\le X}|\widehat{\mu}(m/q)|\ll q^\ell X^{1-\ell}.$$
\end{lemma}

\begin{proof}
    We may assume that $X\ge q$, otherwise, the claim is trivial. We see that
    \begin{equation*}
        \sum_{|m|\le X}|\widehat{\mu}(m/q)|
        \le 1+2\sum_{a=1}^{q}\sum_{n=0}^{[X/q]}|\widehat{\mu}(n+\tfrac{a}{q})|\ll q\sup_{0\le\theta\le 1}\sum_{n=0}^{[X/q]}|\widehat{\mu}(n+\theta)|.
    \end{equation*}
    From \eqref{def:lp-dim-mu-star}, it follows that the above is
    $$\ll q [\tfrac{X}{q}]^{1-\ell} \le q^\ell X^{1-\ell}$$
    as claimed.
\end{proof}

\begin{lemma}\label{lem:reduction-to-aut-form}
    Let $x_0\in\R/\Z$ and $q\in\N$. Assume that there are an $\eta>0$ and $b,d>0$ such that for any $\pi\ni\varphi$ cuspidal or unitary Eisenstein series, we have
    $$\int\varphi(n(x_0+\tfrac{x}{q})a(y))\d\mu(x) \ll q^b|y|^\eta S_{2,d}(\varphi)$$
    as $y\to 0$. Then there exists sufficiently large $\ell>0$ so that for any $\phi\in C_\b^{2\ell}([G])$ one has
    $$\int\phi(n(x_0+\tfrac{x}{q})a(y))\d\mu(x) = \int_{[G]}\phi\d m_{[G]}+ O\left(|y|^\eta q^bS_{\infty,\ell}(\phi)\right)$$
    as $y\to 0$.
\end{lemma}

\begin{proof}
    We start by spectrally decomposing $\phi$ as given in Lemma \ref{lem:spectral-decomposition} at $g=n(x_0+\tfrac{x}{q})a(y)$ and integrate in $x$ with respect to $\mu$. We obtain
    \begin{multline*}
        \int\phi(n(x_0+\tfrac{x}{q})a(y))\d\mu(x) -\int_{[G]}\phi
        =\sum_{\pi\text{ cuspidal}}\sum_{\varphi\in\B(\pi)}\frac{\langle \phi,\varphi\rangle_{[G]}}{\|\varphi\|^2_\pi}\int\varphi(n(x_0+\tfrac{x}{q})a(y))\d\mu(x)
        \\+\int_{\Re(s)=0}\sum_{f\in\B(\I(s))}\frac{\langle \phi,\Eis(f)\rangle_{[G]}}{\|f\|^2_{\I(s)}}\int\Eis(f)(n(x_0+\tfrac{x}{q})a(y))\d\mu(x)\frac{\d s}{4\pi i}.
    \end{multline*}
    The interchanges of above $\mu$ integral with spectral sum and integral are justified via absolute convergence of the spectral decomposition that is uniform in $x$ (which follows from Lemma \ref{lem:l1-norm-bound}) and finiteness of $\mu$. Now the proof follows exactly as the proof of Lemma \ref{lem:spectral-decomposition} where we replace the estimates 
    $$\varphi(g)\ll_g S_{2,d}(\varphi)\quad\text{ by }\quad\int\varphi(n(x_0+\tfrac{x}{q})a(y))\d\mu(x) \ll q^b|y|^\eta S_{2,d}(\varphi),$$
    as given in the assumption.
\end{proof}

The following proposition is a stronger version of Theorem \ref{thm:fractal_withoutbase}, under a slightly stronger assumption.

\begin{prop}\label{prop:fractal}
    Let $x_0\in\R/\Z$, $q\in\N$ and $g_0:=n(x_0)a(q^{-1})$. Let $\mu$ be a Borel probability measure such that $${\Dl}\mu>\frac{39}{64}=0.609375.$$ Then there exist an  $\eta>0$, $b>0$, and $\ell\in\N$ so that for any $\phi\in C_\b^{2\ell}(\Gamma\backslash G)$ we have
\begin{equation*}
    \int\phi\left(g_0n(x)a(y)\right)\d\mu(x)=\int_{[G]}\phi\d m_{[G]}+ O_\mu\left(q^b|y|^{\eta}S_{\infty,\ell}(\phi)\right),
\end{equation*}
as $y\to 0$.
\end{prop}

\begin{proof}
First, note that
$$g_0n(x)a(y)=n(x_0+\tfrac{x}{q})a(y)a(q^{-1}).$$
Let $\pi\ni\varphi$ be a cuspidal representation or a unitary Eisenstein series. We start with the first term of the right-hand side of the equation in Lemma \ref{lem:fourier-transform-basis} with $\varphi$. Applying \eqref{eq:bound-hecke-cusp} with $7/64<\vartheta<\Dl\mu-\frac{1}{2}$, \eqref{eq:bound-hecke-eis}, and \eqref{eq:rapid-decay} with $A=-\frac{1}{2}+\vartheta$ we estimate
\begin{align*}
    \sum_{0\neq|m|\le |y|^{-\sigma}}\frac{\lambda_\pi(|m|)}{\sqrt{|m|}}W_{\varphi}(a(my))e(mx_0)\widehat{\mu}(m/q)
    &\ll S_{2,d}(\varphi)\sum_{|m|\le|y|^{-\sigma}}|m|^{\vartheta-\frac{1}{2}}|my|^{\frac{1}{2}-\vartheta}|\widehat{\mu}(m/q)|\\
    &\ll S_{2,d}(\varphi)|y|^{\frac{1}{2}-\vartheta}\sum_{|m|\le |y|^{-\sigma}}|\widehat{\mu}(m/q)|
\end{align*}
for some $d>0$ independent of $\varphi$ and $q$.

Fix $\frac{1}{2}+\vartheta<b<\Dl\mu$. Then from Lemma \ref{lem:averaging-mu-hat-q} we obtain the last sum above is $O\left(q^b|y|^{-\sigma(1-b)}\right)$. Hence, from Lemma \ref{lem:fourier-transform-basis} we obtain
$$\int\varphi(n(x_0+\tfrac{x}{q})a(y))\d\mu(x)\ll S_{2,d}(\varphi)\left(|y|^{\frac{1}{2}} + q^b |y|^{b-(\frac{1}{2}+\vartheta)-(\sigma-1)(1-b)}\right).$$
Choosing $0<\sigma-1$ sufficiently small we obtain that
$$\int\varphi(n(x_0+\tfrac{x}{q})a(y))\d\mu(x)\ll q^b|y|^{\eta}S_{2,d}(\varphi)$$
for some $\eta>0$.
We apply Lemma \ref{lem:reduction-to-aut-form} with $\phi_q(g):=\phi(g a(q^{-1}))$ to obtain
$$\int\phi(g_0n(x)a(y))\d\mu(x)=\int\phi_q(n(x_0+\tfrac{x}{q})a(y))\d\mu(x) = \int_{[G]}\phi_q\d m_{[G]}+ O\left(|y|^\eta q^bS_{\infty,\ell}(\phi_q)\right).$$
Finally, noting that (see, e.g., \cite[Lemma 3.3 (1)]{khalil2023random})
$$S_{\infty,\ell}(\phi_q)\ll \|\mathrm{Ad}(a(q^{-1}))\|^{\ell} S_{\infty,\ell}(\phi) \ll q^{O(1)}S_{\infty,\ell}(\phi)$$
and by measure invariance
$$\int_{[G]}\phi_q = \int_{[G]}\phi,$$
we conclude the proof.
\end{proof}

\begin{proof}[Proof of Theorem \ref{thm:fractal_withoutbase} and Corollary \ref{cor:convol-infty}]
Proof of Theorem \ref{thm:fractal_withoutbase} follows similarly to the proof of Proposition \ref{prop:fractal}, but with $x_0=0$ and $q=1$, and $\Dl$ replaced by $\dl$. Since the error term does not on depend on $x_0$, Corollary \ref{cor:convol-infty} follows.
\end{proof}

\begin{proof}[Proof of Corollary \ref{cor:convol-2}]
    Note that for any two Borel probability measures $\mu_1$ and $\mu_2$ we have 
    \begin{equation}\label{eq:conv-fourier}
    \widehat{\mu_1\ast\mu_2}(\xi)=\widehat\mu_1(\xi)\widehat\mu_2(\xi).
    \end{equation}
    Hence by Cauchy--Schwarz we have
    $$\sum_{|m|\le X} \vert \widehat{\mu_1\ast\mu_2}(m)\vert\leq \left(\sum_{|m|\le X}\vert\widehat\mu_1(m)\vert^2\right)^{1/2}\left(\sum_{|m|\le X}\widehat\mu_1(m)\vert^2\right)^{1/2}\ll X^{1-\frac{l_1+l_2}{2}}$$
    for any $l_i<\dim_{\ell^2}\mu_i$.
    Since each $\mu_i$ is $s_i$-AD regular, by $\eqref{eqn:l2=s}$ it follows, 
    $$\dl\mu_1\ast\mu_2\geq \frac{s_1+s_2}{2}.$$ The corollary now follows from Theorem \ref{thm:fractal_withoutbase}.
\end{proof}

\begin{proof}[Proof of Corollary \ref{intro_thm1}]

From \cite[Theorem 2.6]{CVY24} we recall that for $D$ in arithmetic progression with $\#D=l$ one has
\begin{equation}\label{eqn:l1}
\dl\mu_{b,D} \geq \frac{\log l}{\log b}-\frac{\log(4+\log(2l))}{\log b}.
\end{equation}  
Suppose $\tfrac{39}{64}<s<1$. Let $b(s)\in\N$ such that for every $b\geq b(s)$ we have
\begin{equation}\label{eqn:b(s)}
    s-\frac{\log (4+\log(2b))}{\log b}> \frac{39}{64}\quad\text{and}\quad b-b^s\geq 2.
\end{equation}
Now fix $b\geq b(s)$ and let $2\le l_1\in (b^s,b)\cap\N$. Let $K_{b,D}$ be such that $D$ is any subset of $\{0,\dots,b-1\}$ with $l_1$ elements which are in an arithmetic progression. Thus
$$\Dh K_{b,D}=\frac{\log l_1}{\log b}=:s_1\geq s.$$
Note that 
\begin{equation*}
s_1-\frac{\log(4+\log(2l_1))}{\log b}
\geq  s_1-s+ s-\frac{\log (4+\log(2b))}{\log b}>\frac{39}{64}.
\end{equation*}
Thus from \eqref{eqn:l1} we have $\dl\mu_{b,D}>\tfrac{39}{64}$, consequently, the proof for $\mu_{b,D}$ follows from Theorem \ref{thm:fractal_withoutbase}. For $\mu_{b,D,x}$
the proof immediately follows from \eqref{eq:haus-shifted-IFS}, by Corollary \ref{cor:convol-infty}, and the proof for $\mu_{b,D}$ as above.
\end{proof}

\begin{remark}
    Obviously, choosing $x\in\R\setminus\Q$ we can make $K_{b,D,x}$ an irrational IFS, thus falls into a realm outside of \cite{khalil2023random}. We remark that also in a very different way we can construct a singular measure, supported on an irrational IFS, that satisfies the conclusion of Theorem \ref{thm:fractal_withoutbase}, as follows.
    Let $\rho$ be sufficiently small such that $(\rho^{-0.609375},\rho^{-1}]\cap\Z_{\geq 2}\neq\varnothing$. Choose $l$ from the above interval. Let $\mu$ be the self-similar measure associated to the IFS $\{\rho (x+i), i=0,\cdots,l-1\}$ and the uniform probability vector. Since this IFS satisfies OSC, $\mu$ is $\frac{-\log l}{\log\rho}$-AD regular (see \S\ref{sec:measure-theory}). Note that $\mu\ast\mu$ is a self-similar measure with (equal) contraction $\rho$; see \cite[pp.809]{strichartz1990selfsimilar}. By \eqref{eqn:l2=s}, $\dl\mu\ast\mu=\dim_{\ell^2}\mu=\frac{-\log l}{\log\rho}>0.609375,$ therefore the corollary follows from Theorem \ref{thm:fractal_withoutbase}. Note that when $\rho^{-1}$ is the (positive) $n$'th root of a positive integer then $\mu\ast\mu$ is not absolutely continuous, since $\widehat\mu\nrightarrow 0$ in this case, which follows from $\widehat\mu(\xi)=\prod_{j=1}^\infty g(\rho^j\xi),\quad g(\xi):=l^{-1}\sum_{i=0}^{l-1}\lambda_ie(i\xi)$; see \cite[Lemma 4.1]{strichartz1990selfsimilar}.
\end{remark}

\begin{lemma}\label{lemma:two dim same}
Let $\mu$ be a self-similar measure whose underlying attractor is a missing digit Cantor set. Then 
    \begin{equation*}
        \dl\mu=\Dl\mu.
    \end{equation*}
Moreover, the same is true when the attractor is a shifted missing digit Cantor set.
\end{lemma}

\begin{proof}
    We first prove when $\mu$ is a self-similar measure on $K_{b,D}$. The proof of this is essentially available in \cite{CVY24}, although was not pointed out. Note that using \eqref{eqn:relation_bet_dimen} and \cite[Theorem 4.2]{CVY24}, the lemma follows upon showing that for all $L\in\N,$
    \begin{equation}\label{eqn:f_L}
       f_{L}:=\frac{-\log(\max_{x}b^{-L}\sum_{i=0}^{b^{L}-1}S_{L}(x+\frac{i}{b^L}))}{\log b^L}\leq \Dl\mu,
    \end{equation} where $S_{L}$ is as defined in \cite[Equation 4.3]{CVY24}. Then by the argument as in \cite[\S 4.1]{CVY24}, it is enough to prove \eqref{eqn:f_L} for $L=1.$ Now in the proof of \cite[Lemma 4.3]{CVY24} authors show that for any $N\in \Z_{\geq 0}$ and $\theta\in\R,$
    $$
    \sum_{m=0}^{b^N-1} \vert\widehat\mu(m+\theta)\vert\leq b^{N(1-f_1)}.
    $$
    Hence $$
   \sup_{0\leq \theta\leq 1}\sum_{m=0}^{b^N-1} \vert \widehat\mu(m+\theta)\vert\leq b^{N(1-f_1)},
    $$ which implies \eqref{eqn:f_L}.

    To see the second part, we note if $\mu$ is a self-similar measure on $K_{b,D}+x_0$ for some $x_0\in\R$ then by uniqueness of the self-similar measures $\mu=\mu_1\ast\delta_{x_0}$ where $\mu_1$ is a self-similar measure on $K_{b,D}$; see \eqref{eq:haus-shifted-IFS}. Thus $|\widehat\mu|=|\widehat\mu_1|$.
\end{proof}

\begin{prop}\label{thm3}
Let $\mu$ be a self-similar measure with the IFS being irreducible and having OSC. If $$\Dl\mu>\frac{39}{64},$$ then for any $\psi$ non-increasing monotonic function, 
    \begin{equation*}
    \mu(\mathcal{W}(\psi))=
1 \quad \text{ if } \quad  \sum \psi(q)=\infty. 
\end{equation*}
\end{prop}

\begin{proof}
    The proposition follows combining Proposition \ref{prop:fractal} and \cite[Theorem 12.1]{khalil2023random}.
\end{proof}

\begin{proof}[Proof of Theorem \ref{intro_thm3}]
Let $\mu$ be a self-similar measure whose underlying attractor is a shifted missing digit 
Cantor set. Therefore the corresponding IFS satisfies OSC and is irreducible. Also by Lemma \ref{lemma:two dim same}, $\dl\mu=\Dl\mu$. Thus the hypothesis of Proposition \ref{prop:fractal} is valid. Hence the proof follows from Proposition \ref{thm3}.
\end{proof}

\textbf{Added in proof}: In the view of \cite{benard2024khintchin}, one can remove the OSC condition and the dimension restriction in the above proposition.

\section{Equidistribution with Pointiwse Fourier Decay}

In this section, we prove Theorem \ref{thm:smooth} whose proof is significantly different than the proof of Theorem \ref{thm:fractal_withoutbase}. In particular, here we crucially need to exploit the cancellation in the sum of $\lambda(m)\widehat{\mu}(m)$ unlike the proof of Theorem \ref{thm:fractal_withoutbase}.

\begin{lemma}\label{lem:error-term-additive-twist}
    There exists a $d>0$ such that for every $\delta>0$ and $0<\eta<\min(\tfrac{1}{2},\delta)$ we have the following. Let $\pi$ be either cuspidal representation or unitary Eisenstein series and $\varphi\in\pi$. Then we have
    $$\sum_{m\neq 0}\frac{\lambda_\pi(|m|)}{|m|^{1+\delta}}W_\varphi(a(my))\ll_\eta |y|^{\eta}S_{2,d}(\varphi)$$
    as $y\to 0$.
\end{lemma}

\begin{proof}
    We use the the bound of Whittaker function \eqref{eq:rapid-decay} with $A=-\eta>-\min(\tfrac{1}{2},\delta)$. Then we bound the sum in the lemma by
    $$\ll |y|^{\eta}S_{2,d}(W_\varphi)\sum_{m\ge 1}\frac{|\lambda_\pi(m)|}{m^{1+\delta-\eta}}.$$
    Thus it suffices to show that the above sum is $\nu_\pi^{O(1)}$.
    To see this, first note that for any $X>1$ applying summation by parts we have
    $$\sum_{m\le X}\frac{|\lambda_\pi(m)|}{m^{1+\delta-\eta}} =X^{-(1+\delta-\eta)}\sum_{m\le X}|\lambda_\pi(m)|+\int_1^X\frac{\sum_{m\le t}|\lambda_\pi(m)|}{t^{2+\delta-\eta}}\d t.$$
    Applying Cauchy--Schwarz and employing \eqref{eq:average-GRC} it follows that
    $$\sum_{m\le t}|\lambda_\pi(m)|\ll_\epsilon t^{1+\epsilon}\nu_\pi^\epsilon.$$
    As $1+\delta-\eta>1$, the proof follows after letting $X\to\infty$.
\end{proof}

Now we will estimate a similar sum as in Lemma \ref{lem:error-term-additive-twist} but $1+\delta$ is replaced by $\tfrac{1}{2}+\delta$, namely,
$$\sum_{m\neq 0}\frac{\lambda_\pi(|m|)}{|m|^{\frac{1}{2}+\delta}}e(m\alpha)W_\varphi(a(my)),\quad\alpha\in\R.$$
To prove a polynomial decay in $y$ for the above sum we can not mimic the proof of Lemma \ref{lem:error-term-additive-twist}, as in this proof the absolute convergence of the Dirichlet series $\sum_{m\neq 0}\frac{\lambda_\pi(|m|)}{|m|^s}$ for $\Re(s)>1$ is crucially used, which will not be available in this case.

We will analyse this sum using a Voronoi-type argument. This argument rather uses the meromorphic properties of the Dirichlet series of the additive twisted Hecke eigenvalues. Such an argument is common in literature when $\pi$ is cuspidal. However, we are unable to find a reference for the same when $\pi$ is Eisenstein. Below we give a proof for the required estimate for all unitary automorphic $\pi$.

As a preparation, we first approximate $\alpha$ by a nonzero rational number $\frac{p}{q}$ with $q\in\N$ in its reduced form so that
\begin{equation}\label{eq:dirichlet}
    |\xi|\le \frac{1}{\sqrt{|y|}q},\quad \xi:=\frac{1}{|y|}\left(\alpha-\frac{p}{q}\right),\quad 1\le q\le |y|^{-\frac{1}{2}}.
\end{equation}
The above is guaranteed because of the Dirichlet's approximation. Finally, we abbreviate $W_\varphi(\cdot n(\xi))$ by $W^\xi$. We start with the following Mellin-theoretic property of $W^\xi$.

\begin{lemma}\label{lem:support-mellin-expansion} 
    If $|\Re(z)|<\tfrac{1}{2}$ and $z$ is away from any pole of $Z(\cdot,W^\xi)$ then 
    $$Z(z,\varepsilon,W^\xi)\ll_N \min\left(\left(\frac{1+|\xi|} {1+|z|}\right)^{N},\left(\frac{1+|\xi|} {1+|z|}\right)^{-2\Re(z)}\right) S_{2,d}(W_\varphi)$$
    for some $d>0$ depending only on $N$.
\end{lemma}

\begin{proof}
    If $\Re(z)> -\tfrac{1}{2}$ then
    $$Z(z,\varepsilon,W^\xi)=\int_{\R^\times}W(a(t))e(t\xi)|t|^z\sgn(t)^\varepsilon\d^\times t$$
    converges absolutely. 
    We integrate by parts the above with $(t\partial_t)^N$ and apply \eqref{eq:rapid-decay} to obtain
    $$Z(z,\varepsilon,W^\xi)\ll_N \left(\frac{1+|\xi|}{1+|z|}\right)^NS_{2,d}(W)$$
    If $z$ is not a pole of $Z(W^\xi)$ then applying local functional equation as in \eqref{eq:local-functional-equation} we obtain
    $$Z(z,\varepsilon,W^\xi)=\gamma(\tfrac{1}{2}-z,\pi)Z(-z,\varepsilon,W^\xi(\cdot w)).$$
    Let $\xi\neq 0$ and note that
    $$wn(\xi)w = n\left(\xi^{-1}\right)a\left(\xi^{-2}\right)k_\xi\mod \R^\times,\quad k_\xi: = \begin{pmatrix}&-1\\1&\xi^{-1}\end{pmatrix}.$$
    For $\Re(z)<\tfrac{1}{2}$ we compute that
    \begin{equation*}
    Z(-z,\varepsilon,W^\xi(\cdot w))=\int_{\R^\times}(k_\xi w\cdot W)\left(a(t)n(\xi^{-1})a(\xi^{-2})\right)|t|^{-z}\sgn(t)^\varepsilon\d^\times t.
    \end{equation*}
    Changing variable $t\mapsto t\xi^2$ we obtain
    \begin{equation}\label{eq:p-adic-iwasawa}
        Z(-z,\varepsilon,W^\xi(\cdot w))=|\xi|^{-2z}Z(-z,\epsilon, W_1^\xi),\quad W_1:=k_\xi w\cdot W.
    \end{equation}
    Here we note that $k_\xi w$ varies over a fixed compact set if $|\xi|\gg 1$. Thus we estimate the above integral by $\ll_N (1+|\xi|)^N(1+|z|)^{-N}S_{2,d}(W)$ doing integration by parts, as before. 
    The lemma now follows applying \eqref{eq:bound-gamma-factor}.
\end{proof}

\begin{lemma}\label{lem:voronoi-cuspidal}
    There exists a $d>0$ such that for every $0<\delta\le\tfrac{1}{2}$ and every $0<\eta<\tfrac{\delta}{2}$ we have the following. Let $\pi$ be a cuspidal representation of $G$ and $\varphi\in\pi$ be an automorphic form. Let $p,q,\xi$ be as in \eqref{eq:dirichlet}. Then
    $$\sum_{m\neq 0}\frac{\lambda_\pi(|m|)}{|m|^{\frac{1}{2}+\delta}}e\left(m\tfrac{p}{q}\right)W^\xi(a(my)) \ll_\eta |y|^\eta S_{2,d}(\varphi)$$
    as $y\to 0$.
\end{lemma}

\begin{proof}
We start with the Mellin expansion of $W^\xi$ given by
$$W^\xi(a(y))=\frac{1}{2}\sum_{\varepsilon\in\{0,1\}}\sgn(y)^{\varepsilon}\int_{\Re(z)=\sigma}Z(z,\varepsilon,W^\xi)|y|^{-z}\frac{\d z}{2\pi i}$$
for some sufficiently positive $\sigma$, where $Z(\cdots)$ is as in \eqref{def:zeta-integral}.
Using this, we write the sum in the lemma as
$$\frac{1}{2}\sum_{\varepsilon\in\{0,1\}}\sgn(y)^{\varepsilon}\int_{\Re(z)=\sigma}Z(z,\varepsilon,W^\xi)|y|^{-z}\sum_{m\neq 0}\frac{\lambda_\pi(|m|)\sgn(m)^\varepsilon e\left(m\tfrac{p}{q}\right)}{|m|^{\frac{1}{2}+z+\delta}}\frac{\d z}{2\pi i}.$$
The above interchange of the $m$-sum and the $z$-integral is justified because the $m$-sum converges absolutely for sufficiently large $\sigma$ (which follows from \eqref{eq:bound-hecke-cusp} and \eqref{eq:bound-hecke-eis}) and $Z(z,\varepsilon,W^\xi)$ decays rapidly in $z$ (which follows from Lemma \ref{lem:support-mellin-expansion}).

Let us consider the $\varepsilon=0$ case to ease the notations; the $\varepsilon=1$ case will be similar. Also, from now on, we will drop $\varepsilon(=0)$ from the notations. Moreover, we only consider the $m>0$ part of the above sum; the $m<0$ part can be treated similarly by replacing $\tfrac{p}{q}$ by its negative.

Correspondingly, we can rewrite the $\varepsilon=0$ summand in the above expression as
\begin{equation}\label{eq:additive-twist-mellin}
    \int_{\Re(z)=\sigma}Z(z,W^\xi)|y|^{-z}L(\tfrac{1}{2}+z+\delta,\tfrac{p}{q},\pi)\frac{\d z}{2\pi i}.
\end{equation}
where we define
$$L\left(z,\tfrac{p}{q},\pi\right):=\sum_{m=1}^\infty\frac{\lambda_\pi(m)e\left(m\tfrac{p}{q}\right)}{m^z}.$$
It follows from \eqref{eq:average-GRC} that the above converges absolutely for $\Re(z)>1$.

As $\pi$ is cuspidal, it is known that $L\left(z,\tfrac{p}{q},\pi\right)$ is entire in $z$; see \cite[Appendix 3-4]{kowalski2002rankin}. On the other hand, $L\left(z,\tfrac{p}{q},\pi\right)$ satisfies the functional equation \cite[eq.(A.10),(A.12),(A.13)]{kowalski2002rankin} of the form
\begin{equation}\label{eq:FE-L-function}
    L\left(z,\tfrac{p}{q},\pi\right) = q^{1-2z}\sum_\pm L\left(1-z,\pm\tfrac{\overline{p}}{q},{\pi}\right)\gamma_\pm(z,\pi)
\end{equation}
for some complex meromorphic functions $\gamma_\pm(z,\pi)$ (given in terms of $C^{\pm}$ as in \cite{kowalski2002rankin}) satisfying
$$\gamma_\pm(z,\pi)\asymp_{\Re(z)} (1+|z|)^{1-2\Re(z)}\nu_\pi^{O_{\Re(z)}(1)}$$
as long as $z$ is a fixed distance away from any pole of $\gamma_\pm(\cdot,\pi)$. The above bound follows from a standard application of Stirling's estimate of $\Gamma$-function; see e.g., \cite[Proposition 9.5-9.6]{cogdell1990arithmetic}. Hence, by the functional equation of the $L$-function and the above bounds of the $\gamma$-factor we have
$$L\left(z,\tfrac{p}{q},\pi\right)\ll_{\Re(z)} \left(q(1+|z|)\right)^{1-2\Re(z)}\nu_\pi^{O_{\Re(z)}(1)},\quad \Re(z)<0.$$
Thus applying Phragm\'en--Lindel\"of convexity principle, we conclude
\begin{equation}\label{eq:convexity-additive-twist-cusp}
L\left(z,\tfrac{p}{q},\pi\right)\ll_\epsilon (q(1+|z|))^{1-\Re(z)+\epsilon}\nu_\pi^{O(1)},\quad 0\le \Re(z)\le 1.
\end{equation}

We first treat the case when $|\xi|\ll 1$. In this case, we shift the contour of \eqref{eq:additive-twist-mellin} to $\sigma=-\tfrac{1}{2}+\epsilon$ without crossing any poles. From Lemma \ref{lem:support-mellin-expansion} and \eqref{eq:convexity-additive-twist-cusp} we estimate this shifted integral by $|y|^{\frac{1}{2}-\epsilon}q^{1-\delta}\ll |y|^{\frac{\delta}{2}-\epsilon}$, which yields the claim in this case.

Now we treat the case when $|\xi|\gg 1$. In this case, we shift the above contour of \eqref{eq:additive-twist-mellin} to $\sigma<-\frac{1}{2}-\delta$. We cross a possible pole of $Z(z, W^\xi)$ when $\Re(z)=-\frac{1}{2}$, whose residue we denote by $\mathcal{R}$. Using the functional equation of the $L$-function as in \eqref{eq:FE-L-function} and $Z(z,W^\xi)$ as in \eqref{eq:local-functional-equation} we rewrite \eqref{eq:additive-twist-mellin} as
\begin{equation}\label{eq:after-voronoi}
\mathcal{R}+\sum_\pm\int_{\Re(z)=\sigma}Z\left(-z,W^\xi(\cdot w)\right)|y|^{-z}\gamma(\tfrac{1}{2}-z,\pi)L(\tfrac{1}{2}-z-\delta,\pm\tfrac{\bar{p}}{q},\pi)q^{-2z-2\delta}\gamma_\pm(\tfrac{1}{2}+z+\delta,\pi)\frac{\d z}{2\pi i}.
\end{equation}
First, we estimate $\mathcal{R}$. From \eqref{eq:local-functional-equation} we see that the poles of $Z(z,W^\xi)$ are of the $\tfrac{1}{2}\pm\mu$ where $\pm\mu$ are the spectral parameters of $\pi$ with $\Re(\mu)=0$. We have
\begin{align*}
\mathcal{R}&=\sum_\pm Z\left(\tfrac{1}{2}\mp\mu,W^\xi(\cdot w)\right)|y|^{-z}\left(\mathrm{Res}_{z=-\frac{1}{2}\pm\mu}\gamma(\tfrac{1}{2}-z,\pi)\right)L(\tfrac{1}{2}+z+\delta,\tfrac{p}{q},\pi)\\
&\ll |y|^{\frac{1}{2}}|\xi|^{\frac{1}{2}+\epsilon}q^{1-\delta}\nu_\pi^{O(1)}\ll |y|^{\frac{\delta}{2}+\epsilon}\nu_\pi^{O(1)}.
\end{align*}
The middle estimate follows from \eqref{eq:convexity-additive-twist-cusp}, \eqref{eq:bound-gamma-factor}, and the fact that
$$Z\left(\tfrac{1}{2}\mp\mu,W^\xi(\cdot w)\right)\ll_{\epsilon}|\xi|^{\frac{1}{2}+\epsilon}\nu_\pi^{O(1)},$$
which follows from \eqref{eq:p-adic-iwasawa}, Lemma \ref{lem:support-mellin-expansion}, and an application of Phragm\'en--Lindel\"of. The last estimate follows from \eqref{eq:dirichlet}.

Now we consider the integral in \eqref{eq:after-voronoi}. Using \eqref{eq:p-adic-iwasawa} and changing variable $z\mapsto -z$ we write this integral as
$$q^{-2\delta}\sum_\pm\sum_{m\neq 0}\frac{\lambda(m)}{|m|^{\frac{1}{2}-\delta}}e\left(\pm m\tfrac{p}{q}\right)\int_{\Re(z)=\sigma}\left|\frac{m}{q^2\xi^2 y}\right|^{-z}Z(z,W_1^\xi)\gamma(\tfrac{1}{2}+z,\pi)\gamma_\pm(\tfrac{1}{2}-z+\delta,\pi)\frac{\d z}{2\pi i}$$
where we choose $\sigma=\frac{1}{2}+\delta+\epsilon$. We claim that the above integral is
$$\ll \left|\frac{m}{q^2\xi^2 y}\right|^{-\sigma}|\xi|^{-2\delta}\nu_\pi^{O(1)}.$$
Thus the total contribution upon applying \eqref{eq:average-GRC} is
$$\ll q^{-2\delta}q^{1+2\delta+2\epsilon}|\xi|^{1+2\delta+2\epsilon}y^{\frac{1}{2}+\delta+\epsilon}|\xi|^{-2\delta}\ll |y|^{\delta-\epsilon}$$
where the last estimate follows from \eqref{eq:dirichlet}. Thus it is enough to show the above claim.

Now we are going to prove the claim that
$$\int_{\R}\int_{\R^\times}W_2(t)e(\xi t)|t|^{ir}\gamma(\tfrac{1}{2}+\sigma+ir,\pi)\gamma_\pm(\tfrac{1}{2}-\sigma+\delta-ir,\pi)\d^\times t\d r\ll |\xi|^{-2\delta}\nu_\pi^{O(1)}$$
where $W_2:=W_1\cdot |\cdot|^\sigma$ satisfying
$$(t\partial_t)^jW_2 \ll_j \nu_\pi^{O(1)}.$$
By a standard stationary phase computation (see e.g., \cite[Lemma 2.4]{godber2013additive}) the inner integral evaluates to
$$|r|^{-\frac{1}{2}}W_2\left(-\tfrac{r}{\xi}\right)e\left(r\log\left|\tfrac{r}{e\xi}\right|\right) +O_N\left((|\xi|+|r|)^{-N}\nu_\pi^M\right)$$ up to a fixed constant. Using the bounds of the gamma factor as in \eqref{eq:bound-gamma-factor} we see that the error term above contributes negligibly. To estimate the main term contribution we first see that
$$\gamma(\tfrac{1}{2}+\sigma+ir,\pi)\gamma_\pm(\tfrac{1}{2}-\sigma+\delta-ir,\pi) = |r|^{-2\delta}g(r)$$
where $g$ is a \emph{flat} function, that is, $\partial^j g(r)\ll_j (1+|r|)^{-j}$ with at most polynomial dependency on $\pi$. This follows from the functional equation $\gamma(1/2+s,\pi)=\gamma(1/2-s,\pi)^{-1}$ and \cite[Lemma 1, \S2.2.5]{blomer2023weyl} (in particular, the above has no oscillatory behaviour). Thus the main term equals
$$\int_\R|r|^{-\frac{1}{2}-2\delta}W_2\left(-\tfrac{r}{\xi}\right)g(r)e\left(r\log\left|\tfrac{r}{e\xi}\right|\right)\d r.$$
Note that trivial bound of the above is $|\xi|^{-2\delta+\frac{1}{2}}$. We get a square-root cancellation after another stationary phase estimate.
This is exactly done in the proof of \cite[Lemma 2.3, Lemma 3.3]{godber2013additive}, applying which we conclude the proof.
\end{proof}

\begin{lemma}\label{lem:additive-twist-cuspidal}
    There exists an $\ell>0$ such that for every $0<\delta\le\tfrac{1}{2}$, and every $0<\eta<\tfrac{\delta}{2}$ we have the following. For any $0\le\alpha<1$ and any $\phi\in C^{2\ell}_\b([G])$ we have
    $$\sum_{\pi\text{ cuspidal}}\sum_{\varphi\in\B(\pi)}\frac{\langle \phi,\varphi\rangle_{[G]}}{\|\varphi\|^2_\pi}\sum_{m\neq 0}\frac{\lambda_\pi(|m|)}{|m|^{\frac{1}{2}+\delta}}W_\varphi(a(my))e(m\alpha)\ll_\eta |y|^{\eta}S_{\infty,\ell}(\phi)$$
    as $y\to 0$.
\end{lemma}

\begin{proof}
    Using unipotent equivariance of $W$ as in \eqref{eq:unipotent-equaivariance} and Dirichlet approximation \eqref{eq:dirichlet} we deduce
    $$W_\varphi(a(my))e(m\alpha)=W^\xi(a(my))e\left(m\tfrac{p}{q}\right).$$
    Thus we write the sum in the Lemma as
    $$\sum_{m\neq 0}\frac{\lambda_\pi(|m|)}{|m|^{\frac{1}{2}+\delta}}W^\xi(a(my))e\left(m\tfrac{p}{q}\right).$$
    We conclude using Lemma \ref{lem:voronoi-cuspidal} that
    $$\sum_{m\neq 0}\frac{\lambda_\pi(|m|)}{|m|^{\frac{1}{2}+\delta}}W_\varphi(a(my))e(m\alpha)\ll |y|^{\eta}S_{2,d}(\varphi).$$
    for cuspidal $\pi$. Working as in the proof of Lemma \ref{lem:spectral-decomposition} we see that 
    $$\sum_{\pi\text{ cuspidal}}\sum_{\varphi\in\B(\pi)}\frac{\langle \phi,\varphi\rangle_{[G]}}{\|\varphi\|^2_\pi}S_{2,d}(\varphi)$$
    converges absolutely and is $O\left(S_{\infty,\ell}(\phi)\right)$.
\end{proof}

\begin{lemma}\label{lem:voronoi-eis}
    There exists a $d>0$ such that for every $0<\delta\le\tfrac{1}{2}$ and every $0<\eta<\tfrac{\delta}{2}$ we have the following. Let $\pi$ be a unitary Eisenstein series with parameter $s\in i\R$ and $\varphi=\Eis(f)$ with $f\in\I(s)$. Let $p,q,\xi$ be as in \eqref{eq:dirichlet}. Then
    $$\sum_{m\neq 0}\frac{\lambda_\pi(|m|)}{|m|^{\frac{1}{2}+\delta}}e\left(m\tfrac{p}{q}\right)W^\xi(a(my)) = \mathcal{M}+ O_\eta\left(|y|^\eta S_{2,d}(\varphi)\right)$$
    where 
    \begin{equation}\label{eq:poles-L-fn}
    \mathcal{M}:=\sum_\pm Z(\tfrac{1}{2}-\delta\pm s,W^\xi)|y|^\delta\left(\sqrt{|y|}q\right)^{-1\mp 2s}\frac{\zeta(1\pm 2s)}{\zeta(1+2s)},
    \end{equation}
    as $y\to 0$.
\end{lemma}

\begin{proof}
Working as in the proof of cuspidal case (with the same conventions and notations), we write the sum in the lemma as
\begin{equation}\label{eq:mellin-additive-eis}
    \int_{\Re(z)=\sigma}|y|^{-z}Z(z,W^\xi)L(\tfrac{1}{2}+z+\delta,\tfrac{p}{q},\pi)\frac{\d z}{2\pi i}
\end{equation}
for some $\sigma>\tfrac{1}{2}-\delta$.
Here from \eqref{eq:formula-hecke-eis} it follows that
$$L\left(z,\tfrac{p}{q},\pi\right)=\frac{D\left(z+s,2s,\tfrac{p}{q}\right)}{\zeta(1+2s)}$$
where 
$$D\left(z,s,\tfrac{p}{q}\right):=\sum_{m=1}^\infty\frac{\tau_s(m)e\left(m\tfrac{p}{q}\right)}{m^z},$$ 
as defined in \cite[pp.228]{bettin2013period}. As before, we first determine holomorphic and growth properties of $L\left(z,\tfrac{p}{q},\pi\right)$ for $0\le\Re(z)\le 1$ and $\Re(s)=0$.

Clearly, for $\Re(z)>1$ we have $D\left(z+s,2s,\tfrac{p}{q}\right)\ll_{\Re(z)}1$. On the other hand, $D$ satisfies a functional equation, namely,
\begin{multline*}
    D\left(z+s,2s,\tfrac{p}{q}\right)=-\frac{2}{q}\left(\frac{q}{2\pi}\right)^{2-2z}\Gamma(1-z+s)\Gamma(1-z-s)\\
    \times\left[\cos(\pi z)D(1-z-s,-2s,-\tfrac{\overline{p}}{q})-\cos(\pi s)D(1-z-s,-2s,\tfrac{\overline{p}}{q})\right],
\end{multline*}
which follows from \cite[Lemma 2]{bettin2013period}. Using Stirling's estimate (see, e.g., \cite[eq.(9.1)]{cogdell1990arithmetic})
$$\Gamma(\sigma+it)\ll_\sigma \exp\left(-\frac{\pi}{2}|t|\right)|\sigma+it|^{\sigma-\frac{1}{2}},\quad \sigma,t\in\R, \sigma\notin\Z_{\le 0}$$
for $\Re(s)=0$ we estimate
$$\Gamma(1-z+s)\Gamma(1-z-s)\ll_{\Re(z)}\exp(-\pi\max(|\Im(z)|,|s|))(1+|z|^2+|s|^2)^{\frac{1}{2}-\Re(z)}.$$
Moreover, using that $\cos(\pi z)\asymp_{\Re(z)}\exp(\pi |\Im(z)|)$, we conclude
$$D\left(z+s,2s,\tfrac{p}{q}\right)\ll (q(1+|z|))^{1-2\Re(z)}(1+|s|)^{O_{\Re(z)}(1)},\quad\Re(z)<0,\Re(s)=0.$$
Thus, using Phragm{\'e}n--Lindel\"of convexity principle and $\zeta(1+s)^{-1}\ll (1+|s|)^\epsilon$ for $\Re(s)=0$ (see \cite[pp.50-51]{Tit86}) we obtain
\begin{equation}\label{eq:convexity-shifted-divisor}
    L\left(z,\tfrac{p}{q},\pi\right)\ll_\epsilon(q(1+|z|))^{1-\Re(z)+\epsilon}\nu_\pi^{O(1)},\quad 0\le \Re(z)\le 1.
\end{equation}
On the other hand, from \cite[Lemma 2]{bettin2013period} we deduce that $L\left(z,\tfrac{p}{q},\pi\right)$ has poles in $z$ exactly at the poles of
$$q^{1-2z}\frac{\zeta(z-s)\zeta(z+s)}{\zeta(1+2s)}$$
which are at $z=1\pm s$ and of order $1$ with residues
$q^{-1\mp 2s}\frac{\zeta(1\pm 2s)}{\zeta(1+2s)}$, respectively. 

Now, as in the proof of the cuspidal case, using holomorphic and decay properties of $Z(z,W^\xi)$ for $\Re(z)>-\frac{1}{2}$ as in Lemma \ref{lem:support-mellin-expansion} we shift the contour of \eqref{eq:mellin-additive-eis} to $\Re(z)=\sigma$ where $-\frac{1}{2}<\sigma<0$. Using Cauchy's theorem we can write \eqref{eq:mellin-additive-eis} as the sum of
$$\mathcal{M}=\sum_\pm Z(\tfrac{1}{2}-\delta\pm s,W^\xi)|y|^\delta\left(\sqrt{|y|}q\right)^{-1\mp 2s}\frac{\zeta(1\pm 2s)}{\zeta(1+2s)},$$
and the shifted contour integral \eqref{eq:mellin-additive-eis}. Working exactly as in the cuspidal case, in particular, using Lemma \ref{lem:support-mellin-expansion} and \eqref{eq:convexity-shifted-divisor}, and choosing $\sigma=-\frac{1}{2}+\epsilon$ we bound the shifted contour integral \eqref{eq:mellin-additive-eis} by $O_\epsilon\left(|y|^{\frac{\delta}{2}-\epsilon}\nu_\pi^{O(1)}\right)$.
\end{proof}

\begin{lemma}\label{lem:additive-twist-eisenstein}
    There exists ann $\ell>0$ such that for every $0<\delta\le\tfrac{1}{2}$ and every $0<\eta<\tfrac{\delta}{2}$ we have the following. For any $0\le\alpha<1$ and any $\phi\in C^{2\ell}_\b([G])$ we have
    $$\int_{\Re(s)=0}\sum_{f\in\B(\I(s))}\frac{\langle \phi,\Eis(f)\rangle_{[G]}}{\|f\|^2_{\I(s)}}\sum_{m\neq 0}\frac{\lambda_s(m)}{|m|^{\frac{1}{2}+\delta}}W_f(a(my))e(m\alpha)\frac{\d s}{2\pi i}\ll_\eta |y|^{\eta}S_{\infty,\ell}(\phi)$$
    as $y\to 0$. Here $\lambda_s$ we mean $\lambda_\pi$ where $\pi$ is the Eisenstein series with parameter $s$.
\end{lemma}

\begin{proof}
    We work as in the proof of Lemma \ref{lem:additive-twist-cuspidal} to reduce to $\alpha=\tfrac{p}{q}$ and $W_f$ to $W^\xi$ where $\xi,p,q$ are as in \eqref{eq:dirichlet}. Then we apply Lemma \ref{lem:voronoi-eis} for $\pi$ Eisenstein to write the sum in the lemma as
    \begin{multline}\label{eq:main-terms}
    |y|^\delta\int_{\Re(s)=0}\left(\sqrt{|y|}q\right)^{-1-2s}H(\tfrac{1}{2}-\delta+s)\frac{\d s}{2\pi i}\\+|y|^\delta\int_{\Re(s)=0}\left(\sqrt{|y|}q\right)^{-1+2s}\frac{\zeta(1-2s)}{\zeta(1+2s)}H(\tfrac{1}{2}-\delta-s)\frac{\d s}{2\pi i}
    \\+O\left(|y|^\eta \int_\R\sum_{f\in\B(\I(i t))}\frac{|\langle \phi,\Eis(f)\rangle_{[G]}|}{\|f\|^2_{\I(it)}}S_{2,d}(\Eis(f))\d t\right),
    \end{multline}
    where
    $$H(z)=H(z;W^\xi,\phi,s):=\sum_{f\in\B(\I(s))}\frac{\langle \phi,\Eis(f)\rangle_{[G]}}{\|f\|^2_{\I(s)}} Z(z,W^\xi).$$
    Working as in the proof of Lemma \ref{lem:spectral-decomposition} we see that the integral-sum in the third summand in \eqref{eq:main-terms} is convergent and consequently, the third summand is $O\left(|y|^\eta S_{\infty,\ell}(\phi)\right)$.

    We now deal with the second summand in \eqref{eq:main-terms}. Using \eqref{eq:functional-eqn-eis} and \eqref{eq:intertwiner-isometry} we write
    \begin{align*}
        \frac{\zeta(1-2s)}{\zeta(1+2s)}H(\tfrac{1}{2}-\delta-s)
        &=\sum_{f\in\B(\I(s))}\frac{\langle \phi,\Eis(M^\ast(s)f)\rangle_{[G]}}{\|M^\ast(s)f\|^2_{\I(-s)}} Z(\tfrac{1}{2}-\delta-s,W^\xi)\\
        &=\sum_{f\in\B(\I(-s))}\frac{\langle \phi,\Eis(f)\rangle_{[G]}}{\|f\|^2_{\I(-s)}} Z(\tfrac{1}{2}-\delta-s,W^\xi)
    \end{align*}
    The last equality follows due to \eqref{eq:whittaker-intertwiner-invariance} and noting that $\{M^\ast(s)f\}_{f\in\B(\I(s)}$ forms an orthogonal basis of $\I(-s)$. Using the above and changing variable $s\mapsto-s$ we see the second summand in \eqref{eq:main-terms} equals
    $$|y|^\delta\int_{\Re(s)=0}\left(\sqrt{|y|}q\right)^{-1-2s}H(\tfrac{1}{2}-\delta+s)\frac{\d s}{2\pi i}$$
    which is the same as the first summand in \eqref{eq:main-terms} on which we focus next.

    We first analytically continue $H(z,..,s)$ in a neighbourhood of $\Re(s)=0$. Note that $\overline{\Eis(f)}=\Eis(\bar{f})$ and $\bar{f}\in\I(-s)$ for $\Re(s)=0$. Thus changing $s$ to $-s$ we give a holomorphic realization of $\langle\phi,\Eis(f)\rangle_{[G]}$. Recalling that $\|f\|_{\I(s)}$ is $s$-independent and meromorphic properties of $\Eis(f)$ we obtain that $\frac{\langle \phi,\Eis(f)\rangle_{[G]}}{\|f\|^2_{\I(s)}}$ is holomorphic in $-\tfrac{1}{2}<\Re(s)\le 0$. On the other hand, recalling holomorphic properties of $Z$ and $W$ we get that $Z(\tfrac{1}{2}-\delta+s)$ is holomorphic in $-\tfrac{1}{2}<\Re(s)\le 0$. Finally, incorporating Remark \ref{rem:non-unitary}, using the bound in Lemma \ref{lem:support-mellin-expansion}, and following the proof of Lemma \ref{lem:spectral-decomposition} we obtain that the sum defining $H(\tfrac{1}{2}-\delta+s)$ converges absolutely in $-\tfrac{1}{2}<\Re(s)\le 0$ and thus defines a holomorphic function in the same region. Moreover, the same proof yields that
    $$H(\tfrac{1}{2}-\delta+s)\ll_{N,\Re(s)}(1+|s|)^{-N}\|\D^\ell\phi\|_\infty,\quad 0\le \Re(s)<\tfrac{1}{2},\quad -\tfrac{1}{2}<\Re(s)\le 0$$
    for some $\ell>0$ depending on $N$.

    Now we shift the contour of the first summand in \eqref{eq:main-terms} to $\Re(s)=-\tfrac{1}{2}+\epsilon$ without crossing any pole. Using the last estimate of $H$ with sufficiently large $N$ we bound the shifted integral by $\ll |y|^{\delta-\epsilon}\|\D^\ell\phi\|_\infty$ completing the proof.
\end{proof}

\begin{proof}[Proof of Theorem \ref{thm:smooth}]
    Without loss of generality we assume $K=1$ and drop the subscripts from $\delta_1\in(0,\tfrac{1}{2}],\alpha_1\in\R,\beta_1\in\C$. Let $\pi$ be either cuspidal or unitary Eisenstein series and $\varphi\in\pi$. Applying absolute convergent Fourier expansion of $\varphi-\varphi_0$, as in \eqref{eq:fourier-expansion}, and bound of $\varphi_0(a(y))$, as in Lemma \ref{lem:constant-term-expression} we write
    $$\int\varphi(n(x)a(y))\d\mu(x) = \sum_{m\neq 0}\frac{\lambda_\pi(|m|}{\sqrt{|m|}}W_\varphi(a(my))\widehat{\mu}(m) + O\left(|y|^{\frac{1}{2}}S_{2,d}(\varphi)\right).$$
    Now employing the expression of $\widehat{\mu}$ as in Theorem \ref{thm:smooth} and applying Lemma \ref{lem:error-term-additive-twist} we deduce that the right hand side above is
    $$\beta\sum_{m\neq 0}\frac{\lambda_\pi(|m|)}{|m|^{\frac{1}{2}+\delta}}W_\varphi(a(my))e(m\alpha) + O\left(|y|^{\eta}S_{2,d}(\varphi)\right).$$
    Starting with the spectral decomposition as in Lemma \ref{lem:spectral-decomposition} then applying Lemma \ref{lem:additive-twist-cuspidal} and Lemma \ref{lem:additive-twist-eisenstein}, and working as in the proof of Lemma \ref{lem:reduction-to-aut-form}, we conclude the proof.
\end{proof}

\subsection{Stationary Phase} 
In this subsection, we prove Corollary \ref{thm: function equi}. We start with the stationary phase estimate of the Fourier transform of $\mu^{w,f}$ as in Corollary \ref{thm: function equi}.

\begin{prop}\label{prop: Stationary phase}
Let $f,w$ and $\mu^{w,f}$ be as in Corollary \ref{thm: function equi}. Let $\{x_i\}_{i=1}^n$ be the set of points in the support of $w$ and let $\{k_i\}_{i=1}^n$ be the multi-set of positive integers so that $f'$ has a zero of order $k_i-1$ at $x_i$. 
    
Then there exist $\{a_{i,j}\}_{1\le i\le n}^{j\in \N}$ only depending on $f$ and $w$ so that for any $N>0$
\begin{equation*}
   \widehat{\mu^{w,f}}(\xi) = \sum_{i=1}^n e\left(\xi f(x_i)\right)\sum_{j=0}^{N-1} a_{i,j} \xi^{-\frac{j+1}{k_i}}+O_{f, w, N}\left(\left(1+|\xi|\right)^{-\frac{N+1}{\max_i{k_i}}}\right), 
\end{equation*}
as $|\xi|\to\infty$.
\end{prop}

\begin{proof}
We first discuss how the above proposition can be reduced to a simpler one.
Modifying the support of $w$ and using \cite[Proposition 1]{stein1993harmonic},
without loss of generality, we can assume $i=1$, i.e., there exists a unique stationary point of $f$ in the support of $w$ (equivalently, $x_1=x_2=\cdots= x_{n}$). Also by changing $f$ with $x\mapsto f(x+x_1)-f(x_1)$, we can also assume without loss of generality, that $x_1=0$ and
\begin{equation*}
   f(0)= f'(0)=\cdots=f^{(k-1)}(0)=0,\quad f^{(k)}(0)\neq 0.
\end{equation*}
To prove Proposition \ref{prop: Stationary phase} it is enough to show that for the above $f$ for any $N>0$ we have
\begin{equation*}
    \widehat{\mu^{w,f}}(\xi) = \sum_{j=0}^{N-1} a_{j} \xi^{-\frac{j+1}{k}}+O_{f, w, N}\left(\left(1+|\xi|\right)^{-\frac{N+1}{k}}\right), 
\end{equation*}
for some $a_j\in\C$. The above follows from \cite[Proposition 3]{stein1993harmonic}.
\end{proof}

\begin{proof}[Proof of Corollary \ref{thm: function equi}]
    The proof follows immediately after applying Theorem \ref{thm:smooth} and Proposition \ref{prop: Stationary phase} with $N=\lfloor\tfrac{\max\{k_i\}}{2}\rfloor$.
\end{proof}

\medskip

{\small
\subsection*{Acknowledgements}
A major discussion took place during the ``Analytic Number Theory'' program at the Institute Mittag-Leffler (IML) where we were in residence and we want to thank IML for extraordinary hospitality and work condition. We also thank University of Michigan, Uppsala University, and University of York, where significant parts of this work were completed, for their superb working condition.  We want to thank Asaf Katz, Dmitry Kleinbock, Manuel L{\"u}thi, Andreas Str{\"o}mbergsson, P{\'e}ter Varj{\'u}, and Barak Weiss for helpful comments on an earlier draft of this paper. SD thanks Han Yu for many interesting discussion and Sam Chow for sending the Mathematica code from \cite{CVY24}. We also thank our family for supporting us while writing the paper. Finally, we also acknowledge very helpful comments from the referees. 
}

\bibliographystyle{abbrv}
\bibliography{cantor.bib}

\end{document}